\documentclass[a4paper,authoryear,12pt]{article}
\usepackage{}
\usepackage{amsmath}
\usepackage{cite}
\usepackage[latin1]{inputenc}
\usepackage{amsmath,amsthm,amssymb}
\usepackage{enumerate}
\usepackage[dvipdfm,colorlinks,linkcolor=blue,hyperindex,pdfstartview=FitH,plainpages=false,CJKbookmarks]{hyperref}

\newcommand{\Rmnum}[1]{\expandafter\@slowromancap\romannumeral #1@}

\newcommand{\Om}{\Omega}

\newcommand{\ds}{\displaystyle}

\makeatother
\title{Boundedness and asymptotically stability to chemotaxis system with competitive kinetics and nonlocal terms}
\author {Guangyu Xu\thanks{Corresponding author: guangyuswu@126.com}\\
\small  College of Mathematics and Statistics, Chongqing University,\\
\small Chongqing, 401331, P.R. China }
\date{}
\newtheorem{theorem}{Theorem}[section]

\newtheorem{lemma}{Lemma}[section]
\newtheorem{remark}{Remark}
\newtheorem{corollary}{Corollary}[section]
\begin{document}
\baselineskip20pt \maketitle
\renewcommand{\theequation}{\arabic{section}.\arabic{equation}}
\catcode`@=11 \@addtoreset{equation}{section} \catcode`@=12
\begin{abstract}
  This paper deals with the solution of following chemotaxis system with competitive kinetics and nonlocal terms
  \begin{eqnarray*}
  \left\{
  \begin{array}{llll}
  u_t=d_1\Delta u-\chi_1\nabla\cdot(u\nabla w)+u\left(a_0-a_1u-a_{2}v-a_3\int_\Om u-a_4\int_\Om v\right), &x\in \Omega, t>0,\\
  v_t=d_2\Delta v-\chi_2\nabla\cdot(v\nabla w)+v\left(b_0-b_1u-b_{2}v-b_3\int_\Om u-b_4\int_\Om v\right), &x\in \Omega, t>0,\\
  w_t=d_3\Delta w-\lambda w+k u+l v, &x\in\Omega, t>0,
  \end{array}
  \right.
  \end{eqnarray*}
  in a smoothly bounded domain $\Omega\subset \mathbb{R}^N, N\geq1$, where $a_0, a_1, a_2, b_0, b_1, b_2>0$ and $a_3, a_4, b_3, b_4\in\mathbb{R}$. The purpose of this paper is to investigate the impact on nonlocal terms of the system, and to find clear conditions on parameters such that the system possesses a unique global bounded solution. Our conclusion quantitatively suggests that the nonlocal competitions can contribute to global and uniformly bounded solutions, while the global cooperations are adverse to boundedness of system. That is:
  \begin{itemize}
    \item If $a_3, a_4, b_3, b_4>0$, i.e., there are nonlocal intraspecific and interspecific competitions, when $N\leq2$, then for any positive parameters the solution of system is globally bounded; when $N\geq3$, suitable large $a_1, b_2$ (local intraspecific competitions) ensure there is no blowup.
    \item If $a_3, a_4, b_3, b_4<0$, i.e., there are globally intraspecific and interspecific cooperations, then for any $N\geq1$, a clear and largeness condition on $a_1, b_2$ is obtained which makes the system admits a boundedness solution.
  \end{itemize}
  Furthermore, we consider the globally asymptotically stability of spatially homogeneous equilibrium with weak and strongly asymmetric competition cases, respectively.
\end{abstract}
{\bf Keywords}: Chemotaxis system, boundedness, asymptotically stability, competitive source, nonlocal source.\\
{\bf AMS(2020) Subject Classification}:  35B44; 35K55; 92C17

\section{Introduction}
\hspace*{\parindent}

Chemotaxis indicates the movement of living organisms in response to certain chemical substances in their environments. This type of movement exists in many biological phenomena such as bacteria aggregation and immune system response. In order to depict the aggregation of Dictyostelium discoideum, Keller and Segel heuristically derived the celebrated mathematical model in \cite{keller1970initiation,keller1971model} at the beginning of the 1970s, which is the so-called KS model. The most outstanding characteristic of KS model is the chemotactic term. The corresponding initial boundary value problems have been used to model not only the mentioned biological processes at the microscopic scale but also population dynamics at the macroscopic scale in the context of life sciences, etc. Roughly speaking, such problems often be classified as parabolic-parabolic systems (see \cite{herrero1997blow,winkler2013finite,winkler2017emergence}) or parabolic-elliptic systems (see \cite{herrero1997finite,herrero1996singularity,jager1992explosions,nagai2001blowup,winkler2010boundedness,winkler2011blow,winkler2018finite}). With the development and complement of analytical theory and mathematical method, the original model has been modified by various scholars with the aim of improving its consistency with biological reality. For example, when the logistic type source term including the consumption of resources around the environment is
taken into account in the model, there are also many works on various central aspects like global existence, lager time behaviors, finite time blow-up, and so on (see  \cite{lin2016global,mimura1996aggregating,tao2015boundedness,tello2007chemotaxis,xiang2018strong} and the references therein). For more historical background, mathematical results, biological significance and more extensive progress on the classical KS model and its variants, we would like to mention the surveys \cite{Bellomo2015Toward,hillen2009user,horstmann20031970}.

It is well known that the situation of single population is rare in the real biological systems, there usually are multiple organisms existing at the same time \cite{horstmann2011generalizing,kavallaris2018multi,pearce2007chemotaxis,wolansky2002multi}. Systems of two biological species and one substance without logistic growth factors have been studied in \cite{biler2013blowup,biler2012blowup,conca2011remarks,espejo2009simultaneous,espejo2012simultaneous,espejo2010note,negreanu2014two} with the motivation that whether multi-species chemotaxis mechanisms can be responsible for processes of cell sorting. The competition of living space and resource between biological individuals in the same or different species is common in ecosystem, so in order to formulate the dynamic behaviors of system more clear, the competitions between different organisms or some other evolutionary behaviors often should be taken into account \cite{hibbing2010bacterial,painter2003modelling,yao2006chemotaxis}. The following two species and one chemical stimulus problem with classical Lotka-Volterra type competition \cite{murray1989mathematical} has been proposed by Tello and Winkler in \cite{tello2012stabilization} and then been considered by many authors,
\begin{equation}\label{2}
\left\{
\begin{array}{ll}
u_t=d_1\Delta u-\chi_{1}\nabla\cdot(u\nabla w)+\mu_{1}u(1-u-\bar a_1v),\\
v_t=d_2\Delta v-\chi_2 \nabla\cdot(v\nabla w)+\mu_{2}v(1-v-\bar a_2u),\\
\tau w_t=d_3\Delta w-\lambda w+k u+l v.
\end{array}
\right.
\end{equation}
When $\tau=0, d_i=1 (i=1,2,3), k=l=1$ and $\bar a_1, \bar a_2\in[0, 1)$, Tello and Winkler \cite{tello2012stabilization} proved that, if $2(\chi_1+\chi_2)+\bar a_1\mu_2<\mu_1$ and $2(\chi_1+\chi_2)+\bar a_2\mu_1<\mu_2$, then the solution of the corresponding initial value problem \eqref{2} exists globally and is bounded and the unique nontrivial spatially homogeneous steady state $(\bar{u}^*, \bar{v}^*, \bar{w}^*)$ of the system, as given by
\begin{equation*}\label{1ge}
\bar{u}^*\equiv\frac{1-\bar a_1}{1-\bar a_1\bar a_2},\ \bar{v}^*\equiv\frac{1-\bar a_2}{1-\bar a_1\bar a_2},\ \ \bar{w}^*\equiv\frac{k}{\lambda}\bar{u}^*+\frac{l}{\lambda}\bar{v}^*
\end{equation*}
is globally asymptotically stable. In case of competitive exclusion, i.e., $\bar a_1>1>\bar a_2\geq0$, Stinner, Tello and Winkler \cite{stinner2014competitive} considered system \eqref{2} with $d_3=l=1$ and proved that all nontrivial solutions will be global in time and exists bounded if $k\frac{\chi_1}{\mu_1}+\frac{\chi_2}{\mu_2}<1$. Moreover, under some extra assumptions on $k, \frac{\chi_1}{\mu_1}$ and $\frac{\chi_2}{\mu_2}$, their results suggest that the bounded solution approaches the homogeneous steady state $(0, 1)$ in which the aggressive subpopulation is at its carrying capacity and the less aggressive species has died out. For more follow-up investigations about boundedness, large time behaviors of solutions one can see \cite{black2016,lin2018new,mizukami2018boundedness,wang2019improvement}. Recently, applying the skills given in \cite{Chemotaxis,Winkler2014How}, we \cite{muxu} obtain a sufficient condition on initial data and parameters such that, for arbitrarily lager positive constant $M$ and some points $(\tilde{x}, \tilde{t})$, the solution of problem \eqref{2} with full-parameters version satisfy $u(\tilde{x}, \tilde{t})+v(\tilde{x}, \tilde{t})>M$, this means that the associated carrying capacity of system \eqref{2} can be exceeded during evolution to an arbitrary extent.

As far as the fully parabolic model is concerned, namely, $\tau=1$, Lin, Mu and Wang \cite{lin2015boundedness} studied problem \eqref{2} and got the boundedness of solutions with $N\geq2$ and some other assumptions on parameters. Bai and Winkler \cite{bai2016equilibration} got the global bounded existence of solution when $N\leq2$. For cases $\bar{a}_1, \bar{a}_2\in (0, 1)$ and $\bar{a}_1\geq 1>\bar{a}_2>0$, they further revealed that the global attractivity property of the corresponding equilibrium is actually inherited if the positive parameters satisfy some conditions, respectively. More recently, the authors in \cite{li2019fully} constructed a condition on parameters such that the solution of problem \eqref{2} with $d_i=1\ (i=1,2,3), N=3$ exists globally and bounded. Li \cite{li2019emergence} proved that the corresponding solution of problem \eqref{2} can exceed any given threshold for suitable large initial data. In additions, the authors in \cite{hu2014global} established the existence of nonconstant positive steady states through bifurcation theory, and obtained some conclusions on the stability of the bifurcating solutions. The formation of time-periodic and stable patterns had been investigated in \cite{Wang2015Time} by Hopf bifurcation analysis, the bifurcation values, spatial profiles and time period associated with these oscillating patterns were obtained therein.

The nonlocal terms of integral type have been attracting considerable attention for building the mathematical models in the context of chemotaxis see for instance \cite{armstrong2006continuum,gerisch2008mathematical,sherratt2009boundedness} and the reference therein. In order to develop a new mathematical model of cancer cell invasion of tissue which focusses on the role of the highly controlled invasion mechanisms involved, Szyma\'{n}ska et al. in \cite{szymanska2009mathematical} applied an integro-differential equation model involving cancer cells. As explained in \cite{szymanska2009mathematical}, due to the fact that individual cells proliferating within the overall tumour cell mass have to compete for nutrients, oxygen and space, so even cancer cells under some conditions are suppressed in their proliferation. They considered this phenomenon by using a logistic growth term. But, assuming ordinary logistic growth may be over-simplification. The logistic type growth means that proliferation of the cells depends on the cells and the extracellular matrix concentration at given point, whereas the proliferation probably actually depends on the cell and extracellular matrix concentration in a local neighbourhood, namely, the immediate surrounding of a cell influences its ability to divide. Hence they included a nonlocal term describing a neighbourhood of a cell that inhibits its proliferation in the model. The nonlocal term is
\begin{equation*}
\mu_1u\left(1-\int_\Om k_{1,1}(x,y)udy-\int_\Om k_{1,2}(x,y)vdy\right),
\end{equation*}
here $\mu_1$ denotes the cancer cell proliferation rate, and $k_{1,1}(x,y), k_{1,2}(x,y)$ are given spatial kernels. The terms $u\int_\Om k_{1,1}(x,y)udy$ and $u\int_\Om k_{1,2}(x,y)vdy$ describe the inhibition of the cells' proliferation caused by the density of surrounding cells and extracellular matrix respectively. Green et al. \cite{green2010non} investigated the effect of hepatocyte-stellate cell interactions on the aggregation process, and they also adopted nonlocal terms to represent cell-cell interactions due to overcrowding or the action of the stellate' processes on hepatocyte, see \cite{green2010non} for detail.

Negreanu and Tello \cite{negreanu2013competitive} considered following competitive system under chemotactic effects with nonlocal terms
\begin{eqnarray}{\label{11}}
\left\{
\begin{array}{llll}
u_t=\Delta u-\chi_1\nabla\cdot(u\nabla w)+u(a_0-a_1u-a_{2}v-a_3\int_\Om u-a_4\int_\Om v),\\
v_t=\Delta v-\chi_2\nabla\cdot(v\nabla w)+v(b_0-b_1u-b_{2}v-b_3\int_\Om u-b_4\int_\Om v),\\
0=\Delta w-\lambda w+k_1 u+k_2 v+f,
\end{array}
\right.
\end{eqnarray}
the forcing term $f$ represents that the chemical substance is also introduced in the system from outside, which is a uniformly bounded and satisfies $f\in C^{\alpha,\beta}(\bar\Om\times[0,\infty))$ for $\alpha>0, \beta\geq1+\frac{\alpha}{2}$ and for some positive constant $C_0$ there holds
\begin{equation*}
\int_0^\infty |\sup_{x\in\Om}f-\inf_{x\in\Om}f|\leq C_0<\infty.
\end{equation*}
The result in \cite[Theorem 0.2]{negreanu2013competitive} suggests that the solution of system \eqref{11} exists bounded and converges to the constant pair $(u^*, v^*)$. Issa et al. in \cite{RachidiAsymptotic} considered full-parameter version of \eqref{11} with $f\equiv0$, and they obtained the boundedness and asymptotic behaviors of solution under different cases by so-called eventual comparison method.

In particular, the following classical parabolic-elliptic type KS system with both local and nonlocal heterogeneous logistic source has been considered in \cite{issa2017dynamics},
\begin{eqnarray*}
\left\{
\begin{array}{llll}
u_t=\Delta u-\chi\nabla\cdot(u\nabla v)+u\left(a_0(x,t)-a_1(x,t)u-a_{2}(x,t)\int_\Om u\right),\\
0=\Delta v+u-v,
\end{array}
\right.
\end{eqnarray*}
here $a_0(x,t), a_1(x,t)$ are nonnegative bounded functions, and $a_2(x,t)$ is a bounded real valued function. The authors showed the local existence and uniqueness of classical solutions to the corresponding initial boundary value problem, and studied systematically the global existence and boundedness of solution, the existence of entire and time periodic positive solution. In addition, aiming to detect the influence of nonlocal term on the behavior of solutions, Bian et al. \cite{bian2018nonlocal} considered the classical parabolic-elliptic type KS system with nonlocal terms as follows,
\begin{eqnarray*}
\left\{
\begin{array}{ll}
u_t=\Delta u-\chi\nabla\cdot(u\nabla v)+u^\alpha\left(1-\int_\Om u^\beta dx\right),\\
0=\Delta v+u-v,
\end{array}
\right.
\end{eqnarray*}
where $\alpha\geq1, \beta>1$. The authors proved that if spatial dimension $N\geq3$, and the parameters satisfy $2\leq\alpha<1+\frac{2\beta}{N}$ or
\begin{equation*}
\alpha<2 \  \ \mbox{and}\ \ \frac{N+2}{N}(2-\alpha)<1+\frac{2\beta}{N}-\alpha,
\end{equation*}
then the system admits a unique global classical solution which is uniformly bounded. This conclusion suggests that the $\beta>\frac{N}{2}$ can prevent chemotactic collapse.

Inspired by the mentioned works, in this paper, we study following chemotaxis system of parabolic-parabolic-parabolic type
\begin{eqnarray}{\label{1}}
\left\{
\begin{array}{llll}
u_t=d_1\Delta u-\chi_1\nabla\cdot(u\nabla w)+u\left(a_0-a_1u-a_{2}v-a_3\int_\Om u-a_4\int_\Om v\right), &x\in \Omega, t>0,\\
v_t=d_2\Delta v-\chi_2\nabla\cdot(v\nabla w)+v\left(b_0-b_1u-b_{2}v-b_3\int_\Om u-b_4\int_\Om v\right), &x\in \Omega, t>0,\\
w_t=d_3\Delta w-\lambda w+k u+l v, &x\in\Omega, t>0,\\
\ds\frac{\partial u}{\partial\nu }=\frac{\partial v}{\partial \nu}=\frac{\partial w}{\partial\nu }=0, &x\in\partial\Omega, t>0,\\
(u, v, w)(x, 0)=(u_0, v_0, w_0), &x\in\Omega,
\end{array}
\right.
\end{eqnarray}
where $\Omega\subset \mathbb{R}^N\ (N\geq1)$ is a bounded domain with smooth boundary $\partial\Omega, \frac{\partial}{\partial\nu}$ denotes the derivative with respect to the outer normal on $\partial\Omega$, the parameters satisfy
\begin{equation}\label{csdy}
d_1, d_2, d_3, \chi_1, \chi_2, a_0, a_1, a_2, b_0, b_1, b_2, \lambda, k, l>0\ \mbox{and}\ a_3, a_4, b_3, b_4\in\mathbb{R}.
\end{equation}
The functions $u=u(x, t)$ and $v=v(x, t)$ are the densities of the two species, $w=w(x, t)$ stands for the concentration of an attractive signal produced by $u$ and $v$, $\Delta\psi, \psi\in\{u, v, w\}$ expresses the random diffusions of species and signals with ratios $d_1, d_2$ and $d_3$, respectively. The cross diffusion terms expresses the advection of species due to chemotaxis, $\chi_i\ (i=1,2)$ are two positive constants which measure the sensitivity of the mobile species to the chemical substances. Terms $k u$ and $lv$ indicate that the mobile species produces the chemical substance as time goes by, in the meantime the signal substance also degrading with tempo $\lambda$, $-a_1u^2, -b_2v^2$ represent the crowding effect which is caused by the local intra-specific competition with peers of the same species, while the other two negative feedback terms $-a_2uv$ and $ -b_1vu$ represent the local interspecific competition between the two species. $a_0, b_0$ measure the intrinsic species growth. When $a_3>0$ and $b_4>0$, they are the strength of nonlocal intraspecific competition between the two species, respectively, $a_4$ and $b_3$ are the strength of nonlocal interspecific competition if $a_3, b_3>0$. Similarly, if $a_3, b_4<0$ and $a_4, b_3<0$, then the species globally intraspecific cooperation and interspecific cooperation, respectively. The initial data $(u_0, v_0, w_0)$ are nonnegative and nontrivial function, which satisfy
\begin{equation}\label{csztj}
u_0, v_0\in C^0(\bar\Om),\ w_0\in W^{1,q}(\bar\Om)\ \mbox{for}\ q>\max\{2, N\}.
\end{equation}

For the two species described in problem \eqref{1} with $a_3, a_4, b_3, b_4>0$, the weak competition regime is when
\begin{equation}\label{rjz}
\frac{b_1+b_3|\Om|}{a_1+a_3|\Om|}<\frac{b_0}{a_0}<\frac{b_2+b_4|\Om|}{a_2+a_4|\Om|},
\end{equation}
and in this case the unique positive constant equilibrium $(u^*, v^*, w^*)$ can be calculated as
\begin{equation}\begin{split}\label{0.77}
& u^* :\equiv\frac{a_0(b_2+b_4|\Om|)-b_0(a_2+a_4|\Om|)}{(b_2+b_4|\Om|)(a_1+a_3|\Om|)-(a_2+a_4|\Om|)(b_1+b_3|\Om|)},\\
& v^* :\equiv\frac{a_0(b_1+b_3|\Om|)-b_0(a_1+a_3|\Om|)}{(a_2+a_4|\Om|)(b_1+b_3|\Om|)-(b_2+b_4|\Om|)(a_1+a_3|\Om|)},\\
& w^* :\equiv\frac{a_0[k(b_2+b_4|\Om|)-l(b_1+b_3|\Om|)]+b_0[l(a_1+a_3|\Om|)-k(a_2+a_4|\Om|)]}{\lambda[(b_2+b_4|\Om|)(a_1+a_3|\Om|)-(a_2+a_4|\Om|)(b_1+b_3|\Om|)]}.
\end{split}\end{equation}
While the strongly asymmetric competition case is when
\begin{equation}\label{qjz}
\frac{b_1+b_3|\Om|}{a_1+a_3|\Om|}<\frac{b_2+b_4|\Om|}{a_2+a_4|\Om|}\leq\frac{b_0}{a_0},
\end{equation}
and the corresponding semi-trivial equilibria is
\begin{equation}\begin{split}\label{0.747}
(u^\star, v^\star, w^\star) :\equiv \left(0,\ \frac{b_0}{b_2+b_4|\Om|},\ \frac{lb_0}{\lambda(b_2+b_4|\Om|)}\right).
\end{split}\end{equation}
The full strong competition regime is when
\begin{equation}\label{haheq}
\frac{b_1+b_3|\Om|}{a_1+a_3|\Om|}>\frac{b_0}{a_0}>\frac{b_2+b_4|\Om|}{a_2+a_4|\Om|}.
\end{equation}

Our main results for present paper are summarized as follows.
\begin{itemize}
  \item [1.] We first build global existence and boundedness properties with all space dimensions $N\geq1$. In particular, when $N\leq2$, if there is no globally cooperates between the species, i.e., $a_3, a_4, b_3, b_4>0$, then the solution of problem \eqref{1} exists boundedness for any sizes of parameters given in \eqref{csdy}; while when $a_3, a_4, b_3, b_4<0$, we find a condition on parameters such that the solution exists boundedness. On the other hand, for the high dimensional case $N>2$, a clear condition on $a_1$ and $b_2$ is obtained, which indicates that suitable large $a_1$ and $b_2$ ensures that the solution is bounded uniformly, see Theorem \ref{s};
  \item [2.] After the boundedness outcome, the second aspect we consider in this paper is the stabilization of constant solution $(u^*, v^*, w^*)$ and $(u^\star, v^\star, w^\star)$, respectively. In the weak competition regime \eqref{rjz}, we shall prove that for any initial values the constant solution $(u^*, v^*, w^*)$ is globally asymptotically stable if $a_1$ and $b_2$ properly big, see Theorem \ref{33}. As for the strongly asymmetric case, the same globally asymptotical stability of $(u^\star, v^\star, w^\star)$ is shown when $b_2$ is suitable large only, see Theorem \ref{1q}.
\end{itemize}

Problem \eqref{1} becomes to system \eqref{2} when
\begin{equation}\label{zhuan}\begin{split}
& a_0=a_1=\mu_1,\ a_2=\mu_1\bar a_1,\\
& b_0=b_2=\mu_2,\ b_1=\mu_2\bar a_2,\\
& a_3=a_4=b_3=b_4=0.
\end{split}\end{equation}
In this case, the weak competition regime \eqref{rjz} reduces to $\bar{a}_1, \bar{a}_2\in (0, 1)$ for system \eqref{2}, and the strongly asymmetric competition case \eqref{qjz} becomes to $\bar{a}_1\geq1>\bar{a}_2>0$, while, the full strong competition regime \eqref{haheq} is equivalent to $\bar{a}_1>1$ and $\bar{a}_2>1$. Then our results mentioned above also imply the corresponding conclusions on system \eqref{2}, which improve and extend some existing results, see the corollaries and remarks underneath our main theorems.

For the globally asymptotical stability properties of system \eqref{1} with full strong competition regime \eqref{haheq}, we leave it as a open question for future research.

The rest of paper is organized as follows. In Section 2, we consider the boundedness of solution to problem \eqref{1}. The main theorem in this aspect, some remarks and the detailed proofs are given in order. In Section 3, we study the stabilization of constant stationary solution, and this section is divided into two parts. The globally asymptotical stability of $(u^*, v^*, w^*)$ under weak competition regime shall be obtained in Subsection 3.1, and in Subsection 3.2, we show the globally asymptotical behaviors of $(u^\star, v^\star, w^\star)$ under strongly asymmetric competition case.

\section{Boundedness of solution}

For any constant $a\in\mathbb{R}$, its positive part and its negative part are given by
\begin{equation*}
(a)_+=\max\{0, a\}\ \ \mbox{and}\ \ (a)_-=\max\{0, -a\}.
\end{equation*}
It should be point out that the regularity of initial values in \eqref{csztj} is enough for us to establish the local existence of solution to problem \eqref{1}, see \cite{bai2016equilibration,winkler2010boundedness}, but we need further following assumption on $w_0(x)$,
\begin{equation}\label{csztj2}
w_0(x)\in W^{2,r}(\Omega)\ \mbox{for}\ r>N.
\end{equation}
which makes Lemma \ref{3.6} below valid then a integral estimate of element $w(x,t)$ can be used to obtain the boundedness of solution.

Now, we give the main conclusion about boundedness for problem \eqref{1}.

\begin{theorem}\label{s}
Let initial data satisfy \eqref{csztj} and \eqref{csztj2} and $p>N$ be a constant. Assume the parameters meet \eqref{csdy}. If one of the following conditions is satisfied
\begin{itemize}
  \item $N\leq 2$ and
  \begin{equation}\label{lazx}
  \min\{a_1, b_2\}>\max\left\{(a_3)_-,\ (b_4)_-,\ (a_4)_-,\ (b_3)_-\right\}|\Om|,
  \end{equation}
  \item or $N\geq3$, \eqref{lazx} holds and
  \begin{equation}\begin{split}\label{tj1ha}
  &a_1>(p-1)\chi_{1}+p[(a_3)_-+(a_4)_-]|\Om|^{\frac{1}{p}}+[(a_3)_-+(b_3)_-]|\Om|^{p}+(2k)^{p+1}(\chi_1+\chi_{2})C_{p},\\
  &b_2>(p-1)\chi_{2}+p[(b_4)_-+(b_3)_-]|\Om|^{\frac{1}{p}}+[(a_4)_-+(b_4)_-]|\Om|^{p}+(2l)^{p+1}(\chi_1+\chi_{2})C_{p},
  \end{split}\end{equation}
\end{itemize}
where $C_p$ is a positive constant which is corresponding to the maximal Sobolev regularity (see Lemma \ref{3.6}), then the solution of problem \eqref{1} exists globally and bounded in $\bar\Om\times(0, +\infty)$.
\end{theorem}

The conclusion of Theorem \ref{s} is applicable for problem \eqref{2}. When \eqref{zhuan} holds, then problem \eqref{1} becomes to system \eqref{2} and \eqref{lazx} holds obviously and our condition \eqref{tj1ha} turns to
\begin{equation}\begin{split}\label{tj1ha2}
  &\mu_1>(p-1)\chi_{1}+(2k)^{p+1}(\chi_1+\chi_{2})C_{p},\\
  &\mu_2>(p-1)\chi_{2}+(2l)^{p+1}(\chi_1+\chi_{2})C_{p},
\end{split}\end{equation}
which gives a relation between the sizes of $\mu_i$ and $\chi_i$ such that the solution of problem \eqref{1} exists globally and uniformly bounded. Hence, Theorem \ref{s} covers the result in \cite[Theorem 1.1]{bai2016equilibration} with case $N\leq 2$. We get following corollary automatically which completes the boundedness property of solution to corresponding initial boundary value problem \eqref{2} for space dimensions $N\geq3$.

\begin{corollary}\label{tl1}
Let $N\geq3, d_1, d_2, d_3, \chi_1, \chi_2, \mu_1, \mu_2, \bar{a}_1, \bar{a}_2, \lambda, k, l$ be positive constants and \eqref{tj1ha2} hold. Then, for any choice of initial value satisfy \eqref{csztj} and \eqref{csztj2}, the solution of system \eqref{2} with homogeneous Neumann boundary condition is bounded in $\bar\Om\times(0, +\infty)$.
\end{corollary}

The authors in \cite{lin2015boundedness} also studied problem \eqref{2} and proved that, if $\Om$ is a smooth bounded convex domain with $N\geq3$ and the parameters satisfy
\begin{equation}\label{sc}
\lambda\geq\frac{1}{2}\ \ \mbox{and}\ \ k=l=1,
\end{equation}
and
\begin{equation*}
\mu_1>\frac{\chi_1N}{4},\ \mu_2>\frac{\chi_2N}{4},
\end{equation*}
as well as
\begin{equation*}
\mu_1+\frac{1}{2}\bar{a}_1\mu_1+\frac{1}{2}\bar{a}_2\mu_2\frac{\chi_1}{\chi_2}>\frac{\chi_1N}{2},\ \ \mu_2+\frac{1}{2}\bar{a}_2\mu_2+\frac{1}{2}\bar{a}_1\mu_1\frac{\chi_2}{\chi_1}>\frac{\chi_1N}{2},
\end{equation*}
then the solution of problem \eqref{2} is bounded. The convexity of domain and assumption \eqref{sc} have been deleted in \cite{li2019fully} with special case $d_1=d_2=d_3=1, N=3$.

\begin{remark}
Corollary \ref{tl1} gives a answer to the open question proposed in \cite[Remark 1.5]{lin2015boundedness}, that is, the convexity of domain is not necessary for the bounded solution of problem \eqref{2} for any $N\geq3$. Moreover, the conclusion extends the outcome obtained in \cite{li2019fully} to more general parameters version of problem \eqref{2} with $N>3$.
\end{remark}

\begin{remark}
We give some comments about the impact on nonlocal terms of problem \eqref{1}. If $a_3, a_4, b_3, b_4>0$, then \eqref{lazx} can be deleted since $a_1, b_2>0$, and \eqref{tj1ha} becomes to
\begin{equation*}\begin{split}
  &a_1>(p-1)\chi_{1}+(2k)^{p+1}(\chi_1+\chi_{2})C_{p},\\
  &b_2>(p-1)\chi_{2}+(2l)^{p+1}(\chi_1+\chi_{2})C_{p}.
\end{split}\end{equation*}
So Theorem \ref{s} tell us that under the influence of local and nonlocal competitive kinetics, if $N\leq2$, the solution of problem \eqref{1} exists bounded for any positive parameters; while for $N\geq3$, we need suitably large strength of the crowding effect caused by the local intraspecific competition with peers of the same species to dominate the actions of two cross diffusion terms. When $a_3, a_4, b_3, b_4<0$, which means that there are globally intraspecific cooperate and interspecific cooperate between two species, then \eqref{lazx} and \eqref{tj1ha} suggest that we need some larger $a_1, b_2$ to ensure the solution existing bounded.
\end{remark}

Throughout the sequel in this section, we aim to prove Theorem \ref{s}. As a preliminary, we state firstly the following local well-posedness of solution to problem \eqref{1}, which can be got by well-established fixed point arguments, see \cite{winkler2010boundedness}.

\begin{lemma}{\label{3.1}}
For any initial value satisfying \eqref{csztj} and parameters meeting \eqref{csdy}, there exists $T_{\max}\in (0, \infty]$ and a uniquely determined triple $(u, v, w)$ of functions
\begin{equation*}
\begin{split}
& u(x,t) \in C^{0}([0,T_{\max})\times\overline{\Omega})\cap C^{2,1}(\overline{\Omega}\times(0,T_{\max})),\\
& v(x,t) \in C^{0}([0,T_{\max})\times\overline{\Omega})\cap C^{2,1}(\overline{\Omega}\times(0,T_{\max})),\\
& w(x,t) \in C^{0}([0,T_{\max})\times\overline{\Omega})\cap C^{2,1}(\overline{\Omega}\times(0,T_{\max}))\cap L^{\infty}_{loc}([0,T_{\max}); W^{1,q}(\Omega)),
\end{split}
\end{equation*}
which solves problem \eqref{1} classically in $\Omega\times(0,T_{\max})$. Furthermore, we have the following extensibility criterion: If $T_{\max}<\infty$, then
\begin{equation*}
\lim_{t\rightarrow T_{\max}}\left(\|u(\cdot, t)\|_{L^{\infty}(\Omega)}+\|v(\cdot, t)\|_{L^{\infty}(\Omega)}+\|w(\cdot, t)\|_{L^{\infty}(\Omega)}\right)=\infty.
\end{equation*}
\end{lemma}

We shall need the following auxiliary lemma \cite{tao2016blow} to derive some time independent estimates.

\begin{lemma}\label{2222}
Let $T>0, \tau\in(0, T), a, b>0$. Suppose that $y: [0, T)\rightarrow [0, \infty)$ is absolutely continuous and
\begin{equation*}
y'(t)+ay(t)\leq h(t)\ \mbox{for a.e.}\ t\in(0, T)
\end{equation*}
with nonnegative function $h\in L_{loc}^1([0, T))$ satisfying
\begin{equation*}
\int_t^{t+\tau} h(s)ds\leq b\ \mbox{for all}\ t\in[0, T-\tau).
\end{equation*}
Then
\begin{equation*}
y(t)\leq\max\left\{y(0)+b,\ \frac{b}{a\tau}+2b\right\}\ \mbox{for all}\ t\in[0, T).
\end{equation*}
\end{lemma}

The following lemma gives us the boundedness of solution with $L^1$-norm.

\begin{lemma}\label{l1yj}
Assume that $(u, v, w)$ is the solution of problem \eqref{1}, $\Om\subset\mathbb{R}^N, N\geq1$ and condition \eqref{lazx} holds. Then there exist positive constants $C_1$ and $C_2$ such that
\begin{equation}\label{yijie}
\int_\Omega(u+v)\leq C_1\ \ \mbox{for all}\ \ t\in(0, T_{\max})
\end{equation}
and
\begin{equation}\label{vds}
\int_t^{t+\tau}\int_\Om(u^2+v^2)\leq C_2\ \ \mbox{for all}\ \ t\in(0, T_{\max}-\tau)
\end{equation}
with $\tau:=\min\{1,\ \frac{1}{2}T_{\max}\}$.
\end{lemma}

\begin{proof}
Integrating the first and the second equations of \eqref{1} over $\Om$ and using the boundary condition we can see
\begin{equation*}\begin{split}
\frac{d}{dt}\int_\Om u=a_0\int_\Omega u-a_1\int_\Omega u^2-a_2\int_\Omega uv-a_3\left(\int_\Om u\right)^2-a_4\int_\Omega u\int_\Om v
\end{split}\end{equation*}
and
\begin{equation*}\begin{split}
\frac{d}{dt}\int_\Om v=b_0\int_\Omega v-b_1\int_\Omega uv-b_2\int_\Omega v^2-b_3\int_\Omega u\int_\Om v-b_4\left(\int_\Om v\right)^2.
\end{split}\end{equation*}
By the Young inequality and the assumptions $a_0, a_1, a_2, b_0, b_1, b_2>0$ we have
\begin{equation}\begin{split}\label{zxc}
\frac{d}{dt}\int_\Om (u+v)&\leq a_0\int_\Omega u+b_0\int_\Omega v-a_1\int_\Omega u^2-b_2\int_\Omega v^2+C_3\left(\int_\Om(u+v)\right)^2,
\end{split}\end{equation}
where
\begin{equation*}
C_3:=\max\left\{(a_3)_-,\ (b_4)_-,\ \frac{(a_4)_-+(b_3)_-}{2}\right\},
\end{equation*}
then \eqref{zxc} and the Cauchy-Schwarz inequality leads to
\begin{equation*}
\frac{d}{dt}\int_\Om (u+v)\leq\max\{a_0, b_0\} \int_\Omega(u+v)-C_4\left(\int_\Om(u+v)\right)^2,
\end{equation*}
where $C_4:=\frac{\min\{a_1, b_2\}}{|\Om|}-C_3>0$ in view of \eqref{lazx}, so we get \eqref{yijie} by a straightforward ODE comparison argument.

Applying \eqref{yijie} to \eqref{zxc} we can see
\begin{equation*}\begin{split}
\frac{d}{dt}\int_\Om (u+v)\leq (a_0+b_0)C_1-a_1\int_\Omega u^2-b_2\int_\Omega v^2+C_3C_1^2.
\end{split}\end{equation*}
Upon a time integration from $t$ to $t+\tau$, by the fact that $u, v$ are nonnegative and $\tau\leq 1$ we obtain
\begin{equation*}
\int_t^{t+\tau}\int_\Om(u^2+v^2)\leq\frac{(a_0+b_0+1)C_1+C_3C_1^2}{\min\{a_1, b_2\}},
\end{equation*}
so we finish our proof.
\end{proof}

\begin{lemma}
Assume that $(u, v, w)$ is the solution of problem \eqref{1} and $\Om\subset\mathbb{R}^N, N\geq1$ and condition \eqref{lazx} holds. Then there exists positive constant $C_6$ such that
\begin{equation}\label{c6}
\int_t^{t+\tau}\int_\Om |\Delta w|^2\leq C_6\ \ \mbox{for all}\ \ t\in(0, T_{\max}-\tau),
\end{equation}
with $\tau:=\min\{1,\ \frac{1}{2}T_{\max}\}$.
\end{lemma}

\begin{proof}
By multiplying the third equation of \eqref{1} with $-\Delta w$ and integrating by parts we see that
\begin{equation}\label{lvc}\begin{split}
\frac{1}{2}\frac{d}{dt}\int_\Omega|\nabla w|^2=-d_3\int_\Omega|\Delta w|^2-\lambda\int_\Omega|\nabla w|^2-k\int_\Omega u\Delta w-l\int_\Omega v\Delta w.
\end{split}\end{equation}
Using the Young inequality, we get
\begin{equation*}
\frac{d}{dt}\int_\Omega|\nabla w|^2+2\lambda\int_\Omega|\nabla w|^2\leq\frac{\max\{k^2, l^2\}}{d_3}\int_\Omega(u^2+v^2).
\end{equation*}
Applying Lemma \ref{2222} with $y(t)=\int_\Omega|\nabla w|^2, a=2\lambda, b=\frac{\max\{k^2, l^2\}}{d_3}C_2$, here $C_2$ is the positive constant given in \eqref{vds}, then above inequality deduces
\begin{equation}\label{c7}
\int_\Omega |\nabla w|^2\leq C_7\ \ \mbox{for all}\ \ t\in(0, T_{\max}).
\end{equation}
We can infer from \eqref{lvc} that
\begin{equation*}
\begin{split}
\frac{1}{2}\frac{d}{dt}\int_\Omega|\nabla w|^2+\frac{d_3}{2}\int_\Omega|\Delta w|^2\leq\frac{\max\{k^2, l^2\}}{d_3}\int_\Omega(u^2+v^2).
\end{split}\end{equation*}
Upon a time integration from $t$ to $t+\tau$, by \eqref{c7}, \eqref{vds} and the fact $\tau\leq 1$ we obtain \eqref{c6}.
\end{proof}

The following key lemma will be used to prove global existence and boundedness of solution.

\begin{lemma}{\label{3.12}}
For the solution $(u, v, w)$ of problem (\ref{1}) with $N\geq1$ and condition \eqref{lazx} holds, if there exists $p_{0}\geq1$ such that $p_0>\frac{N}{2}$ and
\begin{equation*}
\|u(\cdot, t)\|_{L^{p_{0}}(\Omega)}+\|w(\cdot, t)\|_{L^{p_{0}}(\Omega)}<+\infty \,\,\,\ \text{for all}\,\,\,\ t\in[0, T_{\max}).
\end{equation*}
Then $T_{\max}=+\infty$ and we can find a constant $C>0$ independent of $t$ such that
\begin{equation*}
\begin{split}
\sup_{t>0}\left(\|u(\cdot, t)\|_{L^{\infty}(\Omega)}+\|v(\cdot, t)\|_{L^{\infty}(\Omega)}+\|w(\cdot, t)\|_{L^{\infty}(\Omega)}\right)\leq C.
\end{split}
\end{equation*}
Moreover, the bounded solution is a H\"{o}lder function in the sense that there exists $\sigma \in(0,1)$ such that
\begin{equation*}
\|u\|_{C^{2+\sigma, 1+\frac{\sigma}{2}}(\overline{\Omega}\times[t,t+1])}+\|v\|_{C^{2+\sigma, 1+\frac{\sigma}{2}}(\overline{\Omega}\times[t,t+1])}+\|w\|_{C^{2+\sigma, 1+\frac{\sigma}{2}}(\overline{\Omega}\times[t,t+1])}\leq C
\end{equation*}
for all $t\geq1$.
\end{lemma}

\begin{proof}
In the first equation of \eqref{1}, by the nonnegativity of $a_2v$ and \eqref{yijie} we know
\begin{equation*}\begin{split}
u\left(a_0-a_1u-a_2v-a_3\int_\Om u-a_4\int_\Om v\right)&\leq u\left[a_0-a_1u+(a_3)_-C_1+(a_4)_-C_1\right]\\
&\leq\frac{\left[a_0+(a_3)_-C_1+(a_4)_-C_1\right]^2}{4a_1}.
\end{split}\end{equation*}
Similarly, we can control the last term in the second equation of \eqref{1} by a positive constant which is independent to $t$. Then the remaind proof is similar to that in \cite[Lemma 2.6]{bai2016equilibration}, to avoid repetition it is not described here.
\end{proof}

In order to use Lemma \ref{3.12} we further need some bounded estimates of solution with $L^p(\Om)$ norm. The following Gagliardo-Nirenberg inequality \cite{nirenberg1966extended} is useful for us.

\begin{lemma}{\label{3.t10}}
Suppose that $\kappa_1\in(0,\kappa_2)$, $1\leq\kappa_2, \kappa_3\leq\infty$ with $(N-\kappa_3)\kappa_2<N\kappa_3$. Then for any $\psi\in W^{1,\kappa_3}(\Omega)\cap L^{\kappa_1}(\Omega)$, there exists optimum constant $C_{GN}>0$ such that
\begin{equation*}
\begin{split}
||\psi||_{L^{\kappa_2}(\Omega)}\leq C_{GN}\cdot\left(||\nabla \psi||^{\lambda^{\ast}}_{L^{\kappa_3}(\Omega)}\cdot||\psi||^{1-\lambda^{\ast}}_{L^{\kappa_1}(\Omega)}+||\psi||_{L^{\kappa_1}(\Omega)}\right),
\end{split}
\end{equation*}
where
\begin{equation*}{\label{3.11}}
\begin{split}
\lambda^{\ast}=\frac{\frac{N}{\kappa_1}-\frac{N}{\kappa_2}}{1-\frac{N}{\kappa_3}+\frac{N}{\kappa_1}}\in(0,1).
\end{split}
\end{equation*}
\end{lemma}

Applying the maximal Sobolev regularity (see \cite[Theorem 3.1]{matthias1997heat}), we next give a integral estimate of solution for the Neumann problem to inhomogeneous linear heat equation (see e.g. \cite[Lemma 2.1]{ishida2014boundedness}, \cite[Lemma 2.2]{li2018large} and \cite[Lemma 2.2]{yang2015boundedness}), which relates to the third equation in (\ref{1}), and we will use this result to solution element $w(x,t)$.

\begin{lemma}{\label{3.6}}
Let $r>1$ be a constant. Consider the following initial-boundary value problem
\begin{eqnarray}{\label{3.7}}
\left\{
\begin{array}{llll}
\zeta_t=d_3\Delta \zeta-\lambda\zeta+g,\quad &(x,t)\in\Omega\times(0,T),\\
\frac{\partial\zeta}{\partial\nu}=0,\quad &(x,t)\in\partial\Omega\times(0,T),\\
\zeta(x, 0)=\zeta_0(x),\quad &x\in\Omega,
\end{array}
\right.
\end{eqnarray}
then for any $\zeta_{0}(x)\in W^{2,r}(\Omega) (r>N)$ with $\frac{\partial\zeta_{0}(x)}{\partial\nu}=0$ on $\partial\Omega$ and $g(x, t)\in L^{r}((0,T); L^{r}(\Omega))$, problem \eqref{3.7} admits a unique classical solution. Moreover, for any $s_{0}\in[0,T)$, there exists constant $C_{r}>0$ such that for all $s\in(s_0, T)$ we have
\begin{equation*}
\begin{split}
&\int_{s_{0}}^{T}e^{\lambda rs}||\Delta \zeta(\cdot,s)||_{L^{r}(\Omega)}^{r}\\
&\leq C_{r}\int_{s_{0}}^{T}e^{\lambda rs}||g(\cdot,s)||_{L^{r}(\Omega)}^{r}+C_{r}e^{\lambda rs_{0}}\left(||\zeta(\cdot,s_{0})||_{L^{r}(\Omega)}^{r}+||\Delta \zeta(\cdot,s_{0})||_{L^{r}(\Omega)}^{r}\right).
\end{split}
\end{equation*}
\end{lemma}

The following lemma gives us the boundedness of solution with $L^2(\Om)$ norm when $N=2$.

\begin{lemma}
Assume that $(u, v, w)$ is the solution of problem \eqref{1}, $\Om\subset\mathbb{R}^2$ and condition \eqref{lazx} holds. Then for all $t\in(0, T_{\max})$ there exists positive constant $C$ such that
\begin{equation}\label{yijie3}
\int_\Omega\left(u^2+v^2\right)\leq C.
\end{equation}
\end{lemma}

\begin{proof}
Testing the first equation of \eqref{1} against $u$ and integrating by parts
\begin{equation*}\begin{split}
\frac{1}{2}\frac{d}{dt}\int_\Om u^2&=-d_1\int_\Omega |\nabla u|^2+\chi_1\int_\Omega u\nabla v\cdot\nabla w+a_0\int_\Om u^2\\
&\quad\ -\int_\Om u^2\left[a_1u+a_2v+a_3\int_\Om u+a_4\int_\Om v\right],
\end{split}\end{equation*}
which togethers with \eqref{yijie} and the assumptions $a_2\geq0$ implies
\begin{equation}\label{9c}
\frac{1}{2}\frac{d}{dt}\int_\Om u^2\leq-d_1\int_\Omega |\nabla u|^2+\chi_1\int_\Omega u\nabla v\cdot\nabla w+\int_\Om u^2\left[a_0+(a_3)_-C_1+(a_4)_-C_1-a_1u\right].
\end{equation}
It is easy to see that
\begin{equation*}
\int_\Om u^2\left[a_0+(a_3)_-C_1+(a_4)_-C_1-a_1u\right]\leq c_1:=\frac{4\left[a_0+(a_3)_-C_1+(a_4)_-C_1\right]^3}{27a_1^2}|\Om|.
\end{equation*}
On the other hand, once more integrating by parts and using the Cauchy-Schwarz inequality we find that
\begin{equation}\label{hae}
\chi_1\int_\Omega u\nabla v\cdot\nabla w=-\frac{\chi_1}{2}\int_\Omega u^2\Delta w\leq\frac{\chi_1}{2}\|u\|^2_{L^4(\Om)}\|\Delta w\|_{L^2(\Om)}.
\end{equation}
Let $\psi=u, \kappa_2=4, \kappa_3=\kappa_1=2$ and $N=2$ in Lemma \ref{3.t10}, then
\begin{equation*}
\|u\|^2_{L^4(\Om)}\leq (2C_{GN})^2\left(\|\nabla u\|_{L^2(\Om)}\|u\|_{L^2(\Om)}+\|u\|_{L^2(\Om)}^2\right),
\end{equation*}
which combines \eqref{hae} and the Young inequality imply that there exists positive constant $c_2$ such that
\begin{equation*}\begin{split}
\chi_1\int_\Omega u\nabla v\cdot\nabla w&\leq2\chi_1C^2_{GN}\left(\|\nabla u\|_{L^2(\Om)}\|u\|_{L^2(\Om)}+\|u\|_{L^2(\Om)}^2\right)\|\Delta w\|_{L^2(\Om)}\\
&\leq d_1\int_\Omega |\nabla u|^2+c_2\|u\|_{L^2(\Om)}^2\left(\|\Delta w\|_{L^2(\Om)}^2+1\right).
\end{split}\end{equation*}
Then, \eqref{9c} gives us
\begin{equation}\label{jif}
\frac{d}{dt}\int_\Om u^2\leq 2c_2\left(\|\Delta w\|_{L^2(\Om)}^2+1\right)\left(\int_\Om u^2\right)+2c_1.
\end{equation}
We next fix $t\in(0, T_{\max})$, and then obtain from \eqref{vds} that there exists $t_0\in [0, T_{\max})$ such that $t-\tau\leq t_0\leq t$ and
\begin{equation*}
\int_\Om u(x, t_0)^2\leq c_3:=\max\left\{C_2, \ \int_\Om u_0^2\right\},
\end{equation*}
so an integration of \eqref{jif} over $(t_0, t)$ shows that
\begin{equation*}
\int_\Om u^2\leq\left(\int_\Om u^2(x, t_0)\right)\cdot e^{2c_2\int_{t_0}^t\left(\|\Delta w\|_{L^2(\Om)}^2+1\right)}+2c_1\int_{t_0}^te^{2c_2\int_{s}^t\left(\|\Delta w\|_{L^2(\Om)}^2+1\right)},
\end{equation*}
which togethers with \eqref{c6} and the fact $t-t_0\leq\tau\leq1$ lead to \eqref{yijie3} immediately.
\end{proof}

For the case $N\geq3$, we have following $L^p(\Om)$-norm estimate for any $p\geq1$.

\begin{lemma}{\label{4.1}}
For any $p>N$ and the solution of problem \eqref{1}, if condition \eqref{tj1ha} holds, then there exists a positive constant $C$ such that for all $t\in(0, T_{\max})$
\begin{equation*}
\int_{\Omega}\left(u^{p}+v^{p}\right)\leq C.
\end{equation*}
\end{lemma}

\begin{proof}
Multiplying the first equation of \eqref{1} by $u^{p-1}$, integrating by parts over $\Omega$ and using the assumptions $a_2\geq0$, \eqref{yijie} and the Young inequality, we obtain
\begin{equation}\label{4.4}
\begin{split}
\frac{1}{p}\frac{d}{dt}\int_{\Omega}u^{p}&\leq-d_1(p-1)\int_{\Omega}u^{p-2}|\nabla u|^{2}+(p-1)\chi_{1}\int_{\Omega}u^{p-1}\nabla u\cdot\nabla w\\
&\quad+\int_{\Omega}u^{p}\left(a_0-a_1u-a_3\int_\Om u-a_4\int_\Om v\right)\\
&\leq-\frac{(p-1)\chi_{1}}{p}\int_{\Omega}u^{p}\Delta w+\int_{\Omega}u^{p}\left(a_0-a_1u-a_3\int_\Om u-a_4\int_\Om v\right)\\
&\leq\frac{(p-1)\chi_{1}}{p}\left[\frac{p}{p+1}\int_{\Omega}u^{p+1}+\frac{1}{p+1}\int_{\Omega}|\Delta w|^{p+1}\right]\\
&\quad+a_0\int_{\Omega}u^{p}-a_1\int_{\Omega}u^{p+1}+(a_3)_-\int_{\Omega}u^{p}\int_\Om u+(a_4)_-\int_{\Omega}u^{p}\int_\Om v.
\end{split}
\end{equation}
By the Young and H\"{o}lder inequality we further know
\begin{equation*}\begin{split}
(a_3)_-\int_{\Omega}u^{p}\int_\Om u&\leq\frac{p(a_3)_-}{p+1}\left(\int_{\Omega}u^{p}\right)^{\frac{p+1}{p}}+\frac{(a_3)_-}{p+1}\left(\int_{\Omega}u\right)^{p+1}\\
&\leq\frac{p(a_3)_-}{p+1}|\Om|^{\frac{1}{p}}\int_{\Omega}u^{p+1}+\frac{(a_3)_-}{p+1}|\Om|^{p}\int_{\Omega}u^{p+1}.
\end{split}\end{equation*}
Similarly,
\begin{equation*}\begin{split}
(a_4)_-\int_{\Omega}u^{p}\int_\Om v\leq\frac{p(a_4)_-}{p+1}|\Om|^{\frac{1}{p}}\int_{\Omega}u^{p+1}+\frac{(a_4)_-}{p+1}|\Om|^{p}\int_{\Omega}v^{p+1}.
\end{split}\end{equation*}
Hence, adding $\lambda\int_{\Omega}u^{p}$ on both sides of (\ref{4.4}) implies
\begin{equation*}
\begin{split}
&\frac{1}{p}\frac{d}{dt}\int_{\Omega}u^{p}+\lambda\int_{\Omega}u^{p}\\
&\leq\frac{(p-1)\chi_{1}}{p(p+1)}\int_{\Omega}|\Delta w|^{p+1}+\left(a_0+\lambda\right)\int_{\Omega}u^{p}\\
&\quad+\frac{1}{p+1}\left((p-1)\chi_{1}+p[(a_3)_-+(a_4)_-]|\Om|^{\frac{1}{p}}+(a_3)_-|\Om|^{p}-a_1\right)\int_{\Omega}u^{p+1}\\
&\quad+\frac{(a_4)_-}{p+1}|\Om|^{p}\int_{\Omega}v^{p+1},
\end{split}
\end{equation*}
then,
\begin{equation}\label{4.5}
\begin{split}
&\frac{d}{dt}\int_{\Omega}u^{p}+p\lambda\int_{\Omega}u^{p}\\
&\leq\chi_{1}\int_{\Omega}|\Delta w|^{p+1}+p\left(a_0+\lambda\right)\int_{\Omega}u^{p}+(a_4)_-|\Om|^{p}\int_{\Omega}v^{p+1}\\
&\quad+\left((p-1)\chi_{1}+p[(a_3)_-+(a_4)_-]|\Om|^{\frac{1}{p}}+(a_3)_-|\Om|^{p}-a_1\right)\int_{\Omega}u^{p+1}.
\end{split}
\end{equation}
For any fixed $s_{0}\in[0,T)$, use the variation of constants formula to (\ref{4.5}) for $s\in[s_0, t]\subset[0,T)$, we derive
\begin{equation}\label{4.6}
\begin{split}
\int_{\Omega}u^{p}&\leq e^{-p\lambda(t-s_{0})}\int_{\Omega}u^{p}(s_{0})+\chi_{1}\int^{t}_{s_{0}}e^{-p\lambda(t-s)}\int_{\Omega}|\Delta w|^{p+1}\\
&\quad+\int^{t}_{s_{0}}e^{-p\lambda(t-s)}\bigg[p\left(a_0+\lambda\right)\int_{\Omega}u^{p}+(a_4)_-|\Om|^{p}\int_{\Omega}v^{p+1}\bigg]\\
&\quad+\left((p-1)\chi_{1}+p[(a_3)_-+(a_4)_-]|\Om|^{\frac{1}{p}}+(a_3)_-|\Om|^{p}-a_1\right)\int^{t}_{s_{0}}e^{-p\lambda(t-s)}\int_{\Omega}u^{p+1}.
\end{split}
\end{equation}
Similarly, for element $v$ we can obtain
\begin{equation}\label{4.7}
\begin{split}
\int_{\Omega}v^{p}&\leq e^{-p\lambda(t-s_{0})}\int_{\Omega}v^{p}(s_{0})+\chi_{2}\int^{t}_{s_{0}}e^{-p\lambda(t-s)}\int_{\Omega}|\Delta w|^{p+1}\\
&\quad+\int^{t}_{s_{0}}e^{-p\lambda(t-s)}\bigg[p\left(b_0+\lambda\right)\int_{\Omega}v^{p}+(b_3)_-|\Om|^{p}\int_{\Omega}u^{p+1}\bigg]\\
&\quad+\left((p-1)\chi_{2}+p[(b_4)_-+(b_3)_-]|\Om|^{\frac{1}{p}}+(b_4)_-|\Om|^{p}-b_2\right)\int^{t}_{s_{0}}e^{-p\lambda(t-s)}\int_{\Omega}v^{p+1}.
\end{split}
\end{equation}
Combining (\ref{4.6}) and (\ref{4.7}),
\begin{equation}\label{4.8}
\begin{split}
&\int_{\Omega}\left(u^{p}+v^{p}\right)\\
&\leq e^{-p\lambda(t-s_{0})}\int_{\Omega}\left(u^{p}(s_{0})+v^{p}(s_{0})\right)+(\chi_1+\chi_{2})\int^{t}_{s_{0}}e^{-p\lambda(t-s)}\int_{\Omega}|\Delta w|^{p+1}\\
&\quad+\int^{t}_{s_{0}}e^{-p\lambda(t-s)}\bigg[p\left(a_0+\lambda\right)\int_{\Omega}u^{p}+p\left(b_0+\lambda\right)\int_{\Omega}v^{p}\bigg]\\
&\quad+\left((p-1)\chi_{1}+p[(a_3)_-+(a_4)_-]|\Om|^{\frac{1}{p}}+[(a_3)_-+(b_3)_-]|\Om|^{p}-a_1\right)\int^{t}_{s_{0}}e^{-p\lambda(t-s)}
\int_{\Omega}u^{p+1}\\
&\quad+\left((p-1)\chi_{2}+p[(b_4)_-+(b_3)_-]|\Om|^{\frac{1}{p}}+[(a_4)_-+(b_4)_-]|\Om|^{p}-b_2\right)\int^{t}_{s_{0}}e^{-p\lambda(t-s)}
\int_{\Omega}v^{p+1}.
\end{split}
\end{equation}
Let $\zeta=w, r=p>N$ and $g=ku+lv$ in Lemma \ref{3.6}, we know there exists $C_{p+1}>0$ such that
\begin{equation}\label{4.9}
\begin{split}
&(\chi_1+\chi_{2})\int^{t}_{s_{0}}e^{-p\lambda(t-s)}\int_{\Omega}|\Delta w|^{p+1}\\
&\leq(\chi_1+\chi_{2})C_{p}\int^{t}_{s_{0}}\int_{\Omega}e^{-p\lambda(t-s)}(ku+lv)^{p+1}\\
&\quad +(\chi_1+\chi_{2})C_{p}e^{p\lambda s_{0}}\left(\|w(\cdot,s_{0})\|_{L^{p+1}(\Omega)}^{p+1}+\|\Delta w(\cdot,s_{0})\|_{L^{p+1}(\Omega)}^{p+1}\right).
\end{split}
\end{equation}
Taking (\ref{4.9}) into (\ref{4.8}), we have
\begin{equation}\label{4.11}
\begin{split}
&\int_{\Omega}u^{p}+\int_{\Omega}v^{p}\\
&\leq e^{-p\lambda(t-s_{0})}\int_{\Omega}\left(u^{p}(s_{0})+v^{p}(s_{0})\right)\\
&\quad+(\chi_1+\chi_{2})C_{p+1}e^{p\lambda s_{0}}\left(\|w(\cdot,s_{0})\|_{L^{p+1}(\Omega)}^{p+1}+\|\Delta w(\cdot,s_{0})\|_{L^{p+1}(\Omega)}^{p+1}\right)\\
&\quad+\int^{t}_{s_{0}}e^{-p\lambda(t-s)}\bigg[p\left(a_0+\lambda\right)\int_{\Omega}u^{p}+p\left(b_0+\lambda\right)\int_{\Omega}v^{p}\bigg]\\
&\quad+\int^{t}_{s_{0}}e^{-p\lambda(t-s)}\bigg[K_1\int_{\Omega}u^{p+1}+K_2\int_{\Omega}v^{p+1}\bigg],
\end{split}
\end{equation}
where
\begin{equation*}\begin{split}
K_1&:=(p-1)\chi_{1}+p[(a_3)_-+(a_4)_-]|\Om|^{\frac{1}{p}}+[(a_3)_-+(b_3)_-]|\Om|^{p}-a_1+(\chi_1+\chi_{2})C_{p}(2k)^{p+1},\\
K_2&:=(p-1)\chi_{2}+p[(b_4)_-+(b_3)_-]|\Om|^{\frac{1}{p}}+[(a_4)_-+(b_4)_-]|\Om|^{p}-b_2+(\chi_1+\chi_{2})C_{p}(2l)^{p+1}.
\end{split}\end{equation*}
By application of the Young inequality we know that there exist arbitrarily small and positive constants $\tau_1, \tau_2$, and respectively, positive constants $\bar C(\tau_1)$ and $\tilde{C}(\tau_2)$ such that
\begin{equation*}\begin{split}
p\left(a_0+\lambda\right)\int_{\Omega}u^{p}&\leq \tau_1\int_{\Omega}u^{p+1}+\bar C(\tau_1)|\Omega|,\\
p\left(b_0+\lambda\right)\int_{\Omega}v^{p}&\leq \tau_2\int_{\Omega}v^{p+1}+\tilde{C}(\tau_2)|\Omega|,
\end{split}\end{equation*}
so \eqref{4.11} becomes
\begin{equation*}\begin{split}
&\int_{\Omega}u^{p}+\int_{\Omega}v^{p}\\
&\leq e^{-p\lambda(t-s_{0})}\int_{\Omega}\left(u^{p}(s_{0})+v^{p}(s_{0})\right)\\
&\quad+(\chi_1+\chi_{2})C_{p}e^{p\lambda s_{0}}\left(\|w(\cdot,s_{0})\|_{L^{p+1}(\Omega)}^{p+1}+\|\Delta w(\cdot,s_{0})\|_{L^{p+1}(\Omega)}^{p+1}\right)\\
&\quad+\int^{t}_{s_{0}}e^{-p\lambda(t-s)}\left[\left(K_1+\tau_1\right)\int_{\Omega}u^{p+1}+\left(K_2+\tau_2\right)\int_{\Omega}v^{p+1}\right]\\
&\quad+\left(\bar{C}(\tau_1)+\tilde{C}(\tau_2)\right)|\Om|\int^{t}_{s_{0}}e^{-p\lambda(t-s)}.
\end{split}\end{equation*}
It follows from the condition \eqref{tj1ha} that $K_1+\tau_1\leq0$ and $K_2+\tau_2\leq0$. We further note here that integral $\int^{t}_{s_{0}}e^{-p\lambda(t-s)}$ can be controlled by a positive constant which is independent to $t$. Hence, above inequality completes the proof immediately.
\end{proof}

Now, we are ready to prove Theorem \ref{s}.

\begin{proof}[Proof of Theorem \ref{s}]
For the case $N=1$, by \eqref{yijie} and Lemma \ref{3.12} we know the corresponding solution is bounded. We further know from \eqref{yijie3} and Lemma \ref{3.12} that the solution of problem \eqref{1} is globally bounded for $N=2$. Similarly, for the solution with $N\geq3$, the conclusion also holds due to Lemma \ref{4.1} and Lemma \ref{3.12}.
\end{proof}

\section{Stability of solutions}

In this section, we consider the stability of constant stationary solution to problem \eqref{1}. We first state the globally asymptotically stability of the constant solution $(u^*, v^*, w^*)$ to problem \eqref{1} under the weak competition case.

\begin{theorem}\label{33}
Assume $a_i, b_i>0, i=3, 4$, $(u^*, v^*, w^*)$ be given by \eqref{0.77} and \eqref{rjz} hold. Let
\begin{equation}\label{mzx}\begin{split}
\varpi_1&:=a_1-a_4|\Om|-b_3|\Om|,\\
\varpi_2&:=b_2-a_4|\Om|-b_3|\Om|,\\
\varpi_3&:=\frac{d_1a_2\chi_2^2v^*+d_2b_1\chi_1^2u^*}{16d_1d_2d_3a_2b_1\lambda}.
\end{split}\end{equation}
If
\begin{equation}\label{432}
b_1l^2\varpi_1+a_2k^2\varpi_2>2a_2b_1kl,
\end{equation}
and
\begin{equation}\label{tjas}
\varpi_1\varpi_2>a_2b_1+\left(b_1l^2\varpi_1+a_2k^2\varpi_2-2a_2b_1kl\right)\varpi_3,
\end{equation}
then $(u^*, v^*, w^*)$ is globally asymptotically stable with respect to problem \eqref{1} in the sense that
\begin{equation*}
\begin{split}
\lim_{t\rightarrow +\infty}\left(\|u(\cdot,t)-u^*\|_{L^{\infty}(\Omega)}+\|v(\cdot,t)-v^*\|_{L^{\infty}(\Omega)}+\|w(\cdot,t)-w^*\|_{L^{\infty}(\Omega)}\right)=0.
\end{split}
\end{equation*}
\end{theorem}

\begin{remark}\label{er2}
Our conditions in above theorem make sense with suitable lager $a_1$ and $b_2$. We recall the expressions of $u^*$ and $v^*$ which given by \eqref{0.77}, then
\begin{equation*}\begin{split}
u^* :&\equiv\frac{a_0(b_2+b_4|\Om|)-b_0(a_2+a_4|\Om|)}{(b_2+b_4|\Om|)(a_1+a_3|\Om|)-(a_2+a_4|\Om|)(b_1+b_3|\Om|)}\\
&=\frac{a_0-\frac{b_0(a_2+a_4|\Om|)}{(b_2+b_4|\Om|)}}{(a_1+a_3|\Om|)-\frac{(a_2+a_4|\Om|)(b_1+b_3|\Om|)}{(b_2+b_4|\Om|)}}
\end{split}\end{equation*}
and
\begin{equation*}\begin{split}
v^* :&\equiv\frac{a_0(b_1+b_3|\Om|)-b_0(a_1+a_3|\Om|)}{(a_2+a_4|\Om|)(b_1+b_3|\Om|)-(b_2+b_4|\Om|)(a_1+a_3|\Om|)}\\
&=\frac{\frac{a_0(b_1+b_3|\Om|)}{(a_1+a_3|\Om|)}-b_0}{\frac{(a_2+a_4|\Om|)(b_1+b_3|\Om|)}{(a_1+a_3|\Om|)}-(b_2+b_4|\Om|)}.
\end{split}\end{equation*}
The expressions above tell us that, for fixed $a_0, a_2, a_3, a_4, b_0, b_1, b_3, b_4$, constant $u^*$ and $v^*$ sufficiently close to 0 as $a_1, b_2$ sufficiently large, then for fixed $d_1, d_2, d_3, \chi_1, \chi_2, \lambda$, the constant $\varpi_3$ tends to 0 as $\varpi_1$ and $\varpi_2$ sufficiently big. This implies \eqref{432} and \eqref{tjas} immediately. Subsequently, \eqref{rjz} is obvious.
\end{remark}

When \eqref{zhuan} holds, problem \eqref{1} reduces to system \eqref{2} and \eqref{rjz} reduces to $\bar{a}_1, \bar{a}_2\in (0, 1)$, then the assumptions in Theorem \ref{33} can be rewritten as $\varpi_1=\mu_1, \varpi_2=\mu_2$ and
\begin{equation*}
\varpi_3=\frac{d_1\mu_1\bar{a}_1\chi_2^2v^*+d_2\mu_2\bar{a}_2\chi_1^2u^*}{16d_1d_2d_3\mu_1\mu_2\bar{a}_1\bar{a}_2\lambda},
\end{equation*}
as well as
\begin{equation}\label{432sa}
\bar{a}_1k^2+\bar{a}_2l^2>2\bar{a}_1\bar{a}_2kl,
\end{equation}
and
\begin{equation}\label{tjassa}
1>\bar{a}_1\bar{a}_2+\left(\bar{a}_1k^2+\bar{a}_2l^2-2\bar{a}_1\bar{a}_2kl\right)
\frac{d_1\mu_1\bar{a}_1\chi_2^2v^*+d_2\mu_2\bar{a}_2\chi_1^2u^*}{16d_1d_2d_3\mu_1\mu_2\bar{a}_1\bar{a}_2\lambda}.
\end{equation}
Since $\bar{a}_1, \bar{a}_2\in(0, 1)$, we know \eqref{432sa} holds naturally and \eqref{tjassa} makes sense if $\mu_1, \mu_2$ suitable lager or $\chi_1, \chi_2$ appropriate small. Moreover, \eqref{tjassa} is equivalent to the result obtained in \cite[Theorem 1.2]{bai2016equilibration}.

Next, we study the globally asymptotically behaviors of $(u^\star, v^\star, w^\star)$ with strongly asymmetric exclusion case.

\begin{theorem}\label{1q}
Let $a_i, b_i>0, i=3, 4$, $(u^\star, v^\star, w^\star)$ be given by \eqref{0.747} and \eqref{qjz} hold. Assume $\varpi_1$ and $\varpi_2$ be the constants given by \eqref{mzx} with $\varpi_1>0$, and
\begin{equation*}
\bar{\varpi}_3:=\frac{\chi_2^2v^\star}{16d_2d_3b_1\lambda}.
\end{equation*}
If the parameters satisfy \eqref{432} and
\begin{equation}\label{cdsa}
\varpi_1\varpi_2>a_2b_1+\left(b_1l^2\varpi_1+a_2k^2\varpi_2-2a_2b_1kl\right)\bar{\varpi}_3.
\end{equation}
Then $(u^\star, v^\star, w^\star)$ is globally asymptotically stable with respect to problem \eqref{1} in the sense that
\begin{equation*}
\begin{split}
\lim_{t\rightarrow +\infty}\left(\|u(\cdot,t)-u^\star\|_{L^{\infty}(\Omega)}+\|v(\cdot,t)-v^\star\|_{L^{\infty}(\Omega)}+\|w(\cdot,t)-w^\star\|_{L^{\infty}(\Omega)}\right)=0.
\end{split}
\end{equation*}
\end{theorem}

\begin{remark}\label{r1}
Compared with the weak competition case in Theorem \ref{33}, the strongly asymmetric exclusion case only need appropriate big $b_2$ and fixed $a_1$ which satisfies $a_1>a_4|\Om|+b_3|\Om|$ to stabilize the constant steady-state solution $(u^\star, v^\star, w^\star)$.
\end{remark}

\subsection{Weak competition}

In this subsection we will work towards the proof of Theorems \ref{33}, the arguments are based on an energy-type inequality, see \cite[Lemma 3.2]{bai2016equilibration}.

\begin{lemma}{\label{2.12}}
Suppose the parameters satisfy \eqref{rjz}, \eqref{mzx}, \eqref{432} and \eqref{tjas}. Let $(u, v, w)$ be the global and bounded classical solution of problem \eqref{1} with nonnegative initial value $u_0, v_0\not\equiv0$. Then for all $t>0$ there exist $\delta_1>0$ and $\epsilon_1>0$ such that
\begin{equation}\label{day}
E_1(t)\geq0
\end{equation}
and
\begin{equation}\label{2.15}
E_1'(t)\leq-\epsilon_1 F_1(t),
\end{equation}
where
\begin{equation*}\begin{split}
E_1(t)&:=\int_{\Omega}\left[u(\cdot, t)-u^*-u^*\ln\frac{u(\cdot, t)}{u^*}\right]+\int_{\Omega}\left[v(\cdot, t)-v^*-v^*\ln\frac{v(\cdot, t)}{v^*}\right]\\
&\quad\ +\frac{\delta}{2}\int_{\Omega}(w(\cdot, t)-w^*)^{2},\\
F_1(t)&:=\int_{\Omega}(u(\cdot, t)-u^*)^{2}+\int_{\Omega}(v(\cdot, t)-v^*)^{2}+\int_{\Omega}(w(\cdot, t)-w^*)^{2}.
\end{split}\end{equation*}
\end{lemma}

\begin{proof}
The inequality \eqref{day} can be obtained with the same way in \cite[Lemma 3.2]{bai2016equilibration}. We only prove \eqref{2.15}. According to \eqref{432} and \eqref{tjas} we have
\begin{equation*}
\frac{\varpi_1\varpi_2-a_2b_1}{b_1l^2\varpi_1+a_2k^2\varpi_2-2a_2b_1kl}>\varpi_3,
\end{equation*}
which combines the definitions of $\varpi_1, \varpi_2$ and $\varpi_3$ suggest that
\begin{equation*}
\frac{4\lambda a_2\left[(a_1-a_4|\Om|-b_3|\Om|)(b_2-b_3|\Om|-a_4|\Om|)-b_1a_2\right]}{(a_1-a_4|\Om|-b_3|\Om|)b_1l^2+(b_2-b_3|\Om|-a_4|\Om|)a_2k^2-2a_2b_1kl}
>\frac{d_1a_2\chi_2^2v^*+d_2b_1\chi_1^2u^*}{4d_1d_2d_3b_1},
\end{equation*}
so there exists a constant $\delta_1$ such that
\begin{equation}\label{3tj}
\delta_1<\frac{4\lambda a_2\left[(a_1-a_4|\Om|-b_3|\Om|)(b_2-b_3|\Om|-a_4|\Om|)-b_1a_2\right]}{(a_1-a_4|\Om|-b_3|\Om|)b_1l^2+(b_2-b_3|\Om|-a_4|\Om|)a_2k^2-2a_2b_1kl}
\end{equation}
and
\begin{equation}\label{4tj}
\delta_1>\frac{d_1a_2\chi_2^2v^*+d_2b_1\chi_1^2u^*}{4d_1d_2d_3b_1}.
\end{equation}
For such fixed $\delta_1$ the functional $E_1(t)$ can be rewritten as
\begin{equation*}
E_1(t)=A_1(t)+\frac{a_2}{b_1}B_1(t)+C_1(t)
\end{equation*}
with
\begin{equation*}\begin{split}
& A_1(t):=\int_{\Omega}\left[u(\cdot, t)-u^*-u^*\ln\frac{u(\cdot, t)}{u^*}\right],\\
& B_1(t):=\int_{\Omega}\left[v(\cdot, t)-v^*-v^*\ln\frac{v(\cdot, t)}{v^*}\right],\\
& C_1(t):=\frac{\delta_1}{2}\int_\Om(w-w^*)^2.
\end{split}\end{equation*}
We test the first equation of \eqref{1} by $1-\frac{u^*}{u}$ and use the boundary conditions to obtain
\begin{equation}{\label{2.4z}}
\begin{split}
\frac{d}{dt}A_1(t)&=-d_1u^*\int_{\Omega}\left|\frac{\nabla u}{u}\right|^2+\chi_1u^*\int_{\Omega}\frac{\nabla u}{u}\cdot\nabla w\\
&\quad+\int_\Om(u-u^*)\left(a_0-a_1u-a_2v-a_3\int_\Om u-a_4\int_\Om v\right).
\end{split}
\end{equation}
For the last term of equality \eqref{2.4z} we use the equality
\begin{equation*}
a_0=a_1u^*+a_2v^*+a_3|\Om|u^*+a_4|\Om|v^*
\end{equation*}
and the Young and the H\"{o}lder inequalities to have
\begin{equation*}\begin{split}
&\int_\Om(u-u^*)\left(a_0-a_1u-a_2v-a_3\int_\Om u-a_4\int_\Om v\right)\\
&=\int_\Om(u-u^*)\left(a_1(u^*-u)+a_2(v^*-v)+a_3\int_\Om (u^*-u)+a_4\int_\Om (v^*-v)\right)\\
&=-a_1\int_\Om(u-u^*)^2-a_2\int_\Om(u-u^*)(v-v^*)-a_3\left(\int_\Om(u-u^*)\right)^2-a_4\int_\Om(u-u^*)\int_\Om(v-v^*)\\
&\leq-\left(a_1-a_4|\Om|\right)\int_\Om(u-u^*)^2-a_2\int_\Om(u-u^*)(v-v^*)+a_4|\Om|\int_\Om(v-v^*)^2,
\end{split}\end{equation*}
which together with \eqref{2.4z} implies
\begin{equation}{\label{2.424}}
\begin{split}
\frac{d}{dt}A_1(t)&\leq-d_1u^*\int_{\Omega}\left|\frac{\nabla u}{u}\right|^2+\chi_1u^*\int_{\Omega}\frac{\nabla u}{u}\cdot\nabla w\\
&\quad-\left(a_1-a_4|\Om|\right)\int_\Om(u-u^*)^2-a_2\int_\Om(u-u^*)(v-v^*)+a_4|\Om|\int_\Om(v-v^*)^2.
\end{split}
\end{equation}
Testing the second equation of \eqref{1} by $1-\frac{v^*}{v}$ and using the equality
\begin{equation*}
b_0=b_1u^*+b_2v^*+b_3|\Om|u^*+b_4|\Om|v^*,
\end{equation*}
we can get following inequality by similarly argument,
\begin{equation}{\label{2.5}}
\begin{split}
\frac{d}{dt}B_1(t)&\leq-d_2v^*\int_{\Omega}\left|\frac{\nabla v}{v}\right|^2+\chi_2v^*\int_{\Omega}\frac{\nabla v}{v}\cdot\nabla w\\
&\quad-\left(b_2-b_3|\Om|\right)\int_\Om(v-v^*)^2-b_1\int_\Om(u-u^*)(v-v^*)+b_3|\Om|\int_\Om(u-u^*)^2.
\end{split}
\end{equation}
Multiplying the third equation of \eqref{1} by $w-w^*$ and by the fact $\lambda w^*=ku^*+lv^*$ we further see
\begin{equation}{\label{2.10}}
\begin{split}
\frac{d}{dt}C_1(t)&=-\delta_1 d_3\int_{\Omega}|\nabla w|^{2}-\lambda\delta_1\int_{\Omega}(w-w^*)^2+k\delta_1\int_{\Omega}(w-w^*)(u-u^*)\\
&\quad+l\delta_1\int_{\Omega}(w-w^*)(v-v^*).
\end{split}
\end{equation}
By \eqref{2.424}, \eqref{2.5} and \eqref{2.10}, for all $t>0$ we can see
\begin{equation}\begin{split}{\label{2.11}}
\frac{d}{dt}E_1(t)&\leq-\left(a_1-a_4|\Om|-b_3|\Om|\right)\int_\Om (u-u^*)^2-2a_2\int_\Om (u-u^*)(v-v^*)\\
&\quad+k\delta_1\int_{\Omega}(u-u^*)(w-w^*)-\frac{a_2(b_2-b_3|\Om|-a_4|\Om|)}{b_1}\int_\Om(v-v^*)^2\\
&\quad+l\delta_1\int_{\Omega}(v-v^*)(w-w^*)-\lambda\delta_1\int_{\Omega}(w-w^*)^2-d_1u^*\int_{\Omega}\left|\frac{\nabla u}{u}\right|^2\\
&\quad+\chi_1u^*\int_{\Omega}\frac{\nabla u}{u}\cdot\nabla w-\frac{a_2d_2v^*}{b_1}\int_{\Omega}\left|\frac{\nabla v}{v}\right|^2+\frac{a_2\chi_2v^*}{b_1}\int_{\Omega}\frac{\nabla v}{v}\cdot\nabla w\\
&\quad-\delta_1 d_3\int_{\Omega}|\nabla w|^{2}\\
&=-\int_\Om X\cdot(\mathcal {P}\cdot X^T)-\int_{\Omega}Y\cdot(\mathcal {S}\cdot Y^T),
\end{split}\end{equation}
where $X, Y$ are two vector functions and $X^T, Y^T$ are their transposes, for all $(x,t)\in\Om\times(0,\infty], X, Y$ are defined by
\begin{equation*}\begin{split}
X(x,t)&:=(u-u^*, v-v^*, w-w^*),\\
Y(x,t)&:=\left(\frac{|\nabla u|}{u}, \frac{|\nabla v|}{v}, |\nabla w|\right),
\end{split}\end{equation*}
the constant symmetric matrices $\mathcal {P}$ and $\mathcal {S}$ are defined as follows:
\begin{eqnarray*}
\mathcal {P}:=\left(
\begin{array}{ccc}
a_1-a_4|\Om|-b_3|\Om|\ & a_2\ &-\frac{k\delta_1}{2}\\
a_2\ &\frac{a_2(b_2-b_3|\Om|-a_4|\Om|)}{b_1}\ &-\frac{l\delta_1}{2}\\
-\frac{k\delta_1}{2}\ & -\frac{l\delta_1}{2}\ &\lambda\delta_1
\end{array}
\right),
\end{eqnarray*}
and
\begin{eqnarray*}
\mathcal {S}:=\left(
\begin{array}{ccc}
d_1u^*\ & 0\ &-\frac{\chi_1u^*}{2}\\
0\ &\frac{a_2d_2v^*}{b_1}\ &-\frac{a_2\chi_2v^*}{2b_1}\\
-\frac{\chi_1u^*}{2}\ & -\frac{a_2\chi_2v^*}{2b_1}\ &d_3\delta_1
\end{array}
\right).
\end{eqnarray*}

Next, we claim that $\mathcal {P}$ and $\mathcal {S}$ are two positive definite matrices. In this end, we shall compute the three principal minors of $\mathcal {P}$ and $\mathcal {S}$ respectively, and show that all of them are positive. It is easy to see that $M_1:=a_1-a_4|\Om|-b_3|\Om|>0$ and
\begin{equation*}
M_2:=(a_1-a_4|\Om|-b_3|\Om|)(b_2-b_3|\Om|-a_4|\Om|)\frac{a_2}{b_1}-a_2^2>0
\end{equation*}
due to \eqref{432} and \eqref{tjas}. Moreover,
\begin{equation*}\begin{split}
M_3:&=|\mathcal {P}|=\frac{\delta_1}{4b_1}\bigg(4\lambda\left[(a_1-a_4|\Om|-b_3|\Om|)(b_2-b_3|\Om|-a_4|\Om|)a_2-b_1a_2^2\right]\\
&\quad\quad\quad\quad\quad\quad -\delta_1\left[(a_1-a_4|\Om|-b_3|\Om|)b_1l^2+(b_2-b_3|\Om|-a_4|\Om|)a_2k^2-2a_2b_1kl\right]\bigg)>0
\end{split}\end{equation*}
in view of \eqref{3tj}. Similarly, we can prove that all the three principal minors of $\mathcal {S}$ are positive with \eqref{4tj}. Hence, the claim is true. Furthermore, we can infer from the definiteness properties that there is a suitable positive constant $\epsilon_1$ satisfies
\begin{equation*}
X\cdot(\mathcal {P}\cdot X^T)\geq\epsilon_1|X|^2 \ \  \mbox{and}\ \ Y\cdot(\mathcal {S}\cdot Y^T)\geq\epsilon_1|Y|^2,
\end{equation*}
this combines \eqref{2.11} lead to (\ref{2.15}).
\end{proof}

Beside the relation of $E_1(t)$ and $F_1(t)$, we further need following conclusion which can be found in \cite[Lemma 3.1]{bai2016equilibration}.

\begin{lemma}\label{scs}
Suppose that $f: (1, \infty)\rightarrow [0, \infty)$ is a uniformly continuous nonnegative function such that $\int_1^\infty f(t)dt<\infty$. Then $f(t)\rightarrow 0$ as $t\rightarrow\infty$.
\end{lemma}

Now, we prove our first stabilization result.

\begin{proof}[Proof of Theorem 1.2.]
Making use of \eqref{day} and \eqref{2.15} we obtain
\begin{equation*}
\int_1^\infty F_1(t)dt\leq\frac{E_1(1)}{\epsilon}<\infty.
\end{equation*}
From Lemma \ref{3.12} we know that the bounded solution elements $u, v$ and $w$ are H\"{o}lder continuous in $\bar\Om\times[t; t+1]$ uniformly with respect to $t>1$, from which we conclude that $F_1(t)$ is
uniformly continuous in $(1, \infty)$. Therefore, Lemma \ref{scs} implies
\begin{equation}\label{45}
\int_\Om(u-u_*)^2+\int_\Om(v-v_*)^2+\int_\Om(w-w_*)^2\rightarrow 0,
\end{equation}
as $t\rightarrow\infty$.
By the Gagliardo-Nirenberg inequality we know there exists $C_1>0$ such that
\begin{equation*}
\|u-u_*\|_{L^\infty(\Om)}\leq C_1\|u-u_*\|_{W^{1, \infty}(\Om)}^{\frac{n}{n+2}}\|u-u_*\|_{L^{2}(\Om)}^{\frac{2}{n+2}}
\end{equation*}
Herein, it follows from Lemma 3.2 that $(u(\cdot, t))_{t>1}$ is bounded in $W^{1, \infty}(\Om)$, thus we conclude from \eqref{45} that $u(\cdot, t)\rightarrow u_*$ in $L^\infty(\Om)$ as $t\rightarrow\infty$. By similar arguments for $v$ and $w$ we obtain our conclusions.
\end{proof}

\subsection{Strongly asymmetric competition}

In the sequel, we consider the solution of problem \eqref{1} with exclusion competition case. Similar to the method used in \cite[Lemma 3.4]{bai2016equilibration}, we introduce following two functionals with some positive constant $\delta_2$,
\begin{equation*}\begin{split}
E_2(t)&:=\int_{\Omega}u(\cdot, t)+\int_{\Omega}\left[v(\cdot, t)-v^\star-v^\star\ln\frac{v(\cdot, t)}{v^\star}\right]+\frac{\delta_2}{2}\int_{\Omega}(w(\cdot, t)-w^\star)^{2},\\
F_2(t)&:=\int_{\Omega}u^{2}(\cdot, t)+\int_{\Omega}(v(\cdot, t)-v^\star)^{2}+\int_{\Omega}(w(\cdot, t)-w^\star)^{2}.
\end{split}\end{equation*}

\begin{lemma}{\label{2.122}}
Suppose the parameters satisfy \eqref{432} and \eqref{cdsa}. Let $(u, v, w)$ be the global and bounded classical solution of problem \eqref{1} with $u\not=0$ and $v\not=0$. Then for all $t>0$ there exist $\delta_2>0$ and $\epsilon_2>0$ such that $E_2(t)\geq0$ and
\begin{equation}\label{2.151}
E_2'(t)\leq-\epsilon_2 F_2(t).
\end{equation}
\end{lemma}

\begin{proof}
The proof follows a strategy similar to Lemma \ref{2.12} and the idea comes from \cite[Lemma 3.4]{bai2016equilibration}. For readers' convenience, we give a proof of \eqref{2.151} here. 

It follows from \eqref{432} and \eqref{cdsa} that
\begin{equation*}
\frac{\varpi_1\varpi_2-a_2b_1}{b_1l^2\varpi_1+a_2k^2\varpi_2-2a_2b_1kl}>\bar{\varpi}_3,
\end{equation*}
which leads to
\begin{equation*}
\frac{4\lambda a_2\left[(a_1-a_4|\Om|-b_3|\Om|)(b_2-b_3|\Om|-a_4|\Om|)-b_1a_2\right]}{(a_1-a_4|\Om|-b_3|\Om|)b_1l^2+(b_2-b_3|\Om|-a_4|\Om|)a_2k^2-2a_2b_1kl}
>\frac{a_2\chi_2^2v^\star}{4d_2d_3b_1},
\end{equation*}
so we can find a constant $\delta_2$ satisfies \eqref{3tj} and
\begin{equation}\label{44tj}
\delta_2>\frac{a_2\chi_2^2v^\star}{4d_2d_3b_1}.
\end{equation}
For such fixed $\delta_2$ the functional $E_2(t)$ can be rewritten as
\begin{equation*}
E_2(t)=A_2(t)+\frac{a_2}{b_1}B_2(t)+C_2(t)
\end{equation*}
with
\begin{equation*}\begin{split}
& A_2(t):=\int_{\Omega}u(\cdot, t),\\
& B_2(t):=\int_{\Omega}\left[v(\cdot, t)-v^\star-v^\star\ln\frac{v(\cdot, t)}{v^\star}\right]\\
& C_2(t):=\frac{\delta_2}{2}\int_\Om(w-w^\star)^2.
\end{split}\end{equation*}
Then we can use the first equation of problem \eqref{1} and the Young inequality to obtain
\begin{equation}{\label{2.44}}
\begin{split}
\frac{d}{dt}A_2(t)&=\int_\Om u\left(a_0-a_1u-a_2v-a_3\int_\Om u-a_4\int_\Om v\right)\\
&=\int_\Om u\left(a_0-a_1u-a_2(v-v^\star)-a_2v^\star-a_3\int_\Om u-a_4\int_\Om (v-v^\star)-a_4v^\star|\Om|\right)\\
&=(a_0-a_2v^\star-a_4v^\star|\Om|)\int_\Om u-a_1\int_\Om u^2-a_2\int_\Om u(v-v^\star)-a_3\left(\int_\Om u\right)^2\\
&\quad-a_4\int_\Om u\int_\Om (v-v^\star)\\
&\leq(a_0-a_2v^\star-a_4v^\star|\Om|)\int_\Om u-\left(a_1-a_4|\Om|\right)\int_\Om u^2-a_2\int_\Om u(v-v^\star)\\
&\quad+a_4|\Om|\int_\Om (v-v^\star)^2.
\end{split}\end{equation}
In fact, by the second inequality of \eqref{qjz} we know
\begin{equation*}
v^\star=\frac{b_0}{b_2+b_4|\Om|}>\frac{a_0}{a_2+a_4|\Om|},
\end{equation*}
then it is easy to check that $a_0-a_2v^\star-a_4v^\star|\Om|<0$, so \eqref{2.44} implies
\begin{equation}\label{2.4}
\frac{d}{dt}A_2(t)\leq-\left(a_1-a_4|\Om|\right)\int_\Om u^2-a_2\int_\Om u(v-v^\star)+a_4|\Om|\int_\Om (v-v^\star)^2.
\end{equation}
We testing the second equation of \eqref{1} by $1-\frac{v^\star}{v}$ to obtain
\begin{equation}{\label{2.444}}
\begin{split}
\frac{d}{dt}B_2(t)&=-d_2v^\star\int_{\Omega}\left|\frac{\nabla v}{v}\right|^2+\chi_2v^\star\int_{\Omega}\frac{\nabla v}{v}\cdot\nabla w\\
&\quad+\int_\Om(v-v^\star)\left(b_0-b_1u-b_2v-b_3\int_\Om u-b_4\int_\Om v\right).
\end{split}
\end{equation}
Using the equality $b_0=b_2v^\star+b_4|\Om|v^\star$ and the Young and the H\"{o}lder inequalities to have
\begin{equation*}\begin{split}
&\int_\Om(v-v^\star)\left(b_0-b_1u-b_2v-b_3\int_\Om u-b_4\int_\Om v\right)\\
&=\int_\Om(v-v^\star)\left(b_2v^\star+b_4v^\star|\Om|-b_1u-b_2v-b_3\int_\Om u-b_4\int_\Om v\right)\\
&=\int_\Om(v-v^\star)\left(-b_1u-b_2(v-v^\star)-b_3\int_\Om u-b_4\int_\Om(v-v^\star)\right)\\
&\leq -b_1\int_\Om(v-v^\star)u-(b_2-b_3|\Om|)\int_\Om(v-v^\star)^2+b_3|\Om|\int_\Om u^2.
\end{split}\end{equation*}
Then \eqref{2.444} gives
\begin{equation}{\label{2.52}}
\begin{split}
\frac{d}{dt}B_2(t)&\leq-d_2v^\star\int_{\Omega}\left|\frac{\nabla v}{v}\right|^2+\chi_2v^\star\int_{\Omega}\frac{\nabla v}{v}\cdot\nabla w\\
&\quad -b_1\int_\Om(v-v^\star)u-(b_2-b_3|\Om|)\int_\Om(v-v^\star)^2+b_3|\Om|\int_\Om u^2.
\end{split}
\end{equation}
Multiplying the third equation of \eqref{1} by $w-w^\star$ and by the fact $\lambda w^\star=lv^\star$, we further see
\begin{equation}{\label{2.102}}
\begin{split}
\frac{d}{dt}C_2(t)&=-\delta_2 d_3\int_{\Omega}|\nabla w|^{2}-\lambda\delta_2\int_{\Omega}(w-w^\star)^2+k\delta_2\int_{\Omega}(w-w^\star)u\\
&\quad+l\delta_2\int_{\Omega}(w-w^\star)(v-v^\star).
\end{split}
\end{equation}
It follows from \eqref{2.4}, \eqref{2.52} and \eqref{2.102} that
\begin{equation*}\begin{split}
\frac{d}{dt}E_2(t)&\leq-\left(a_1-a_4|\Om|-b_3|\Om|\right)\int_\Om u^2-2a_2\int_\Om u(v-v^\star)\\
&\quad+k\delta_2\int_{\Omega}u(w-w^\star)-\frac{a_2(b_2-b_3|\Om|-a_4|\Om|)}{b_1}\int_\Om(v-v^\star)^2\\
&\quad+l\delta_2\int_{\Omega}(v-v^\star)(w-w^\star)-\lambda\delta_2\int_{\Omega}(w-w^\star)^2\\
&\quad-\frac{a_2d_2v^\star}{b_1}\int_{\Omega}\left|\frac{\nabla v}{v}\right|^2+\frac{a_2\chi_2v^\star}{b_1}\int_{\Omega}\frac{\nabla v}{v}\cdot\nabla w-\delta_2 d_3\int_{\Omega}|\nabla w|^{2}\\
&=-\int_\Om\bar{X}\cdot(\bar{\mathcal {P}}\cdot \bar{X}^T)-\int_{\Omega}\bar{Y}\cdot(\bar{\mathcal {S}}\cdot \bar{Y}^T),
\end{split}\end{equation*}
where
\begin{equation*}
\bar{X}(x,t):=(u, v-v^\star, w-w^\star)\ \mbox{and}\ \bar{Y}(x,t):=\left(\frac{|\nabla v|}{v}, |\nabla w|\right),
\end{equation*}
furthermore, $\bar{\mathcal {P}}$ and $\bar{\mathcal {S}}$ are still the constant symmetric matrices given by
\begin{eqnarray*}
\bar{\mathcal {P}}:=\left(
\begin{array}{ccc}
a_1-a_4|\Om|-b_3|\Om|\ & a_2\ &-\frac{k\delta_2}{2}\\
a_2\ &\frac{a_2(b_2-b_3|\Om|-a_4|\Om|)}{b_1}\ &-\frac{l\delta_2}{2}\\
-\frac{k\delta_2}{2}\ & -\frac{l\delta_2}{2}\ &\lambda\delta_2
\end{array}
\right),
\end{eqnarray*}
and
\begin{eqnarray*}
\bar{\mathcal {S}}:=\left(
\begin{array}{cc}
\frac{a_2d_2v^\star}{b_1}\ &-\frac{a_2\chi_2v^\star}{2b_1}\\
-\frac{a_2\chi_2v^\star}{2b_1}\ &d_3\delta_2
\end{array}
\right).
\end{eqnarray*}
The positive definiteness of $\bar{\mathcal {P}}$ has been proved in Lemma \ref{2.12} already, and inequality \eqref{44tj} ensures symmetric matrices $\bar{\mathcal {S}}$ is positive definite. Hence, we can get (\ref{2.151}) and finish our proof as in Lemma \ref{2.12}.
\end{proof}

\begin{proof}[Proof of Theorem \ref{1q}]
Theorem \ref{1q} can be proved with similar arguments used in the proof of Theorem \ref{33}, so we omit detail here.
\end{proof}


\begin{thebibliography}{10}
\small

\bibitem{armstrong2006continuum}
N. Armstrong, K. Painter and J. Sherratt,
\newblock A continuum approach to modelling cell-cell adhesion.
\newblock {\em J. Theoret. Biol.}, 243(1): 98--113, 2006.

\bibitem{bai2016equilibration}
X.L. Bai and M. Winkler,
\newblock Equilibration in a fully parabolic two-species chemotaxis system with
  competitive kinetics.
\newblock {\em Indiana Univ. Math. J.}, 65(2): 553--583, 2016.

\bibitem{Bellomo2015Toward}
N.~Bellomo, A.~Bellouquid, Y.~Tao and M.~Winkler,
\newblock Toward a mathematical theory of Keller-Segel models of pattern
  formation in biological tissues.
\newblock {\em Math. Models Methods Appl. Sci.},
  25(09): 1663--1763, 2015.

\bibitem{bian2018nonlocal}
S. Bian, L. Chen and E. Latos,
\newblock Nonlocal nonlinear reaction preventing blow-up in supercritical case
  of chemotaxis system.
\newblock {\em Nonlinear Anal.}, 176: 178--191, 2018.

\bibitem{biler2013blowup}
P. Biler, E. Espejo and I. Guerra,
\newblock Blowup in higher dimensional two species chemotactic systems.
\newblock {\em Comm. Pure Appl. Anal.}, 12(1): 89--98, 2013.

\bibitem{biler2012blowup}
P. Biler and I. Guerra,
\newblock Blowup and self-similar solutions for two-component drift-diffusion
  systems.
\newblock {\em Nonlinear Anal. Theor. Methods Appl.},
  75(13): 5186--5193, 2012.

\bibitem{black2016}
T. Black, J. Lankeit and M. Mizukami,
\newblock On the weakly competitive case in a two-species chemotaxis model.
\newblock {\em IMA J. Appl. Math.}, 81(5): 860--876, 2016.

\bibitem{conca2011remarks}
C. Conca, E. Espejo and K. Vilches,
\newblock Remarks on the blowup and global existence for a two species
  chemotactic Keller-Segel system in $R^2$.
\newblock {\em European J. Appl. Math.}, 22(6): 553--580, 2011.

\bibitem{espejo2009simultaneous}
E. Espejo, A. Stevens and J. Vel{\'a}zquez,
\newblock Simultaneous finite time blow-up in a two-species model for
  chemotaxis.
\newblock {\em Analysis}, 29(3): 317--338, 2009.

\bibitem{espejo2012simultaneous}
E. Espejo, A. Stevens and T. Suzuki.
\newblock Simultaneous blowup and mass separation during collapse in an
  interacting system of chemotactic species.
\newblock {\em Differ. Integral Equ.}, 25(3/4): 251--288, 2012.

\bibitem{espejo2010note}
E. Espejo, A. Stevens and J. Velazquez.
\newblock A note on non-simultaneous blow-up for a drift-diffusion model.
\newblock {\em Differ. Integral Equ.}, 23(5/6): 451--462, 2010.

\bibitem{gerisch2008mathematical}
A.~Gerisch and M. Chaplain.
\newblock Mathematical modelling of cancer cell invasion of tissue: local and
  non-local models and the effect of adhesion.
\newblock {\em J. Theoret. Biol.}, 250(4): 684--704, 2008.

\bibitem{green2010non}
J. Green, S. Waters, J. Whiteley, L. Keshet, K. Shakesheff and H. Byrne,
\newblock Non-local models for the formation of hepatocyte-stellate cell
  aggregates.
\newblock {\em J. Theoret. Biol.}, 267(1): 106--120, 2010.

\bibitem{herrero1997finite}
M. Herrero, E. Medina and J. Vel{\'a}zquez,
\newblock Finite-time aggregation into a single point in a reaction-diffusion
  system.
\newblock {\em Nonlinearity}, 10(6): 1739, 1997.

\bibitem{herrero1996singularity}
M. Herrero and J. Vel{\'a}zquez,
\newblock Singularity patterns in a chemotaxis model.
\newblock {\em Math. Ann.}, 306(1): 583--623, 1996.

\bibitem{herrero1997blow}
M. Herrero and J. Vel{\'a}zquez.
\newblock A blow-up mechanism for a chemotaxis model.
\newblock {\em  Ann. Sc. Norm. Super. Pisa, Cl. Sci.}, 24(4): 633--683, 1997.

\bibitem{hibbing2010bacterial}
M. Hibbing, C. Fuqua, M. Parsek and S. Peterson,
\newblock Bacterial competition: surviving and thriving in the microbial
  jungle.
\newblock {\em Nat. Rev. Microbiol.}, 8(1): 15, 2010.

\bibitem{hillen2009user}
T. Hillen and K. Painter,
\newblock A user's guide to pde models for chemotaxis.
\newblock {\em J. Math. Biol.}, 58(1-2): 183, 2009.

\bibitem{horstmann20031970}
D. Horstmann,
\newblock From 1970 until present: the Keller-Segel model in chemotaxis and its
  consequences I.
\newblock {\em Jahresber. DMV}, 105(3): 103-165, 2003.

\bibitem{horstmann2011generalizing}
D. Horstmann,
\newblock Generalizing the keller--segel model: Lyapunov functionals, steady
  state analysis, and blow-up results for multi-species chemotaxis models in
  the presence of attraction and repulsion between competitive interacting
  species.
\newblock {\em J. Nonlinear Sci.}, 21(2): 231--270, 2011.

\bibitem{hu2014global}
J. Hu, Q. Wang, J.Y. Yang and L. Zhang,
\newblock Global existence and steady states of a two competing species
  Keller-Segel chemotaxis model.
\newblock {\em Kinet. Relat. Models}, 8: 777--807, 2015.

\bibitem{ishida2014boundedness}
S. Ishida, K. Seki and T. Yokota,
\newblock Boundedness in quasilinear Keller-Segel systems of
  parabolic-parabolic type on non-convex bounded domains.
\newblock {\em J. Differential Equations}, 256(8): 2993--3010, 2014.

\bibitem{issa2017dynamics}
T. Issa and W.X. Shen,
\newblock Dynamics in chemotaxis models of parabolic-elliptic type on bounded
  domain with time and space dependent logistic sources.
\newblock {\em SIAM J. Appl. Dyn. Syst.}, 16(2): 926--973,
  2017.

\bibitem{jager1992explosions}
W. J{\"a}ger and S. Luckhaus,
\newblock On explosions of solutions to a system of partial differential
  equations modelling chemotaxis.
\newblock {\em Tran. Amer. Math. Soc.},
  329(2): 819--824, 1992.

\bibitem{Chemotaxis}
L. Johannes,
\newblock Chemotaxis can prevent thresholds on population density.
\newblock {\em Discrete Contin. Dyn. Syst. Ser. B}, 20(5): 1499-1527, 2015.

\bibitem{kavallaris2018multi}
N. Kavallaris, T. Ricciardi and G. Zecca,
\newblock A multi-species chemotaxis system: Lyapunov functionals, duality,
  critical mass.
\newblock {\em European J. Appl. Math.}, 29(3): 515--542, 2018.

\bibitem{keller1970initiation}
E. Keller and L. Segel,
\newblock Initiation of slime mold aggregation viewed as an instability.
\newblock {\em J. Theoret. Biol.}, 26(3): 399--415, 1970.

\bibitem{keller1971model}
E. Keller and L. Segel,
\newblock Model for chemotaxis.
\newblock {\em J. Theoret. Biol.}, 30(2): 225--234, 1971.

\bibitem{li2018large}
J. Li, Y.Y. Ke and Y.F. Wang,
\newblock Large time behavior of solutions to a fully parabolic
  attraction-repulsion chemotaxis system with logistic source.
\newblock {\em Nonlinear Anal.: Real World Appl.}, 39: 261--277, 2018.

\bibitem{li2019fully}
X. Li and Y.L. Wang,
\newblock On a fully parabolic chemotaxis system with Lotka-Volterra
  competitive kinetics.
\newblock {\em J. Math. Anal. Appl.},
  471(1-2): 584--598, 2019.

\bibitem{li2019emergence}
Y. Li,
\newblock Emergence of large densities and simultaneous blow-up in a
  two-species chemotaxis system with competitive kinetics.
\newblock {\em Discrete Contin. Dyn. Syst. Ser. B}, 553--583,
  2019.

\bibitem{lin2016global}
K. Lin and C.L. Mu,
\newblock Global dynamics in a fully parabolic chemotaxis system with logistic
  source.
\newblock {\em Discrete Contin. Dyn. Syst. Ser. A}, 36(9): 5025--5046,
  2016.

\bibitem{lin2015boundedness}
K. Lin, C.L. Mu and L.C. Wang,
\newblock Boundedness in a two-species chemotaxis system.
\newblock {\em Math. Methods Appl. Sci.},
  38(18): 5085--5096, 2015.

\bibitem{lin2018new}
K. Lin, C.L. Mu and H. Zhong,
\newblock A new approach toward stabilization in a two-species chemotaxis model
  with logistic source.
\newblock {\em Comput. Math. Appl.}, 75(3): 837--849,
  2018.

\bibitem{matthias1997heat}
H. Matthias and P. Jan,
\newblock Heat kernels and maximal $L^p-L^q$ estimates for parabolic evolution
  equations.
\newblock {\em Comm. Partial Differential Equations},
  22(9-10): 1647--1669, 1997.

\bibitem{mimura1996aggregating}
M. Mimura and T. Tsujikawa.
\newblock Aggregating pattern dynamics in a chemotaxis model including growth.
\newblock {\em Physica A},
  230(3-4): 499--543, 1996.

\bibitem{mizukami2018boundedness}
M. Mizukami,
\newblock Boundedness and stabilization in a two-species chemotaxis-competition
  system of parabolic-parabolic-elliptic type.
\newblock {\em Math. Methods Appl. Sci.}, 41(1): 234--249,
  2018.

\bibitem{murray1989mathematical}
J. Murray,
\newblock Mathematical biology, 2nd ed., biomathematics, vol. 19, 1993.

\bibitem{nagai2001blowup}
T. Nagai,
\newblock Blowup of nonradial solutions to parabolic-elliptic systems modeling
  chemotaxis in two-dimensional domains.
\newblock {\em J. Inequal. Appl.}, 6: 37--55, 2001.

\bibitem{negreanu2013competitive}
M. Negreanu and J. Tello,
\newblock On a competitive system under chemotactic effects with non-local
  terms.
\newblock {\em Nonlinearity}, 26(4): 1083, 2013.

\bibitem{negreanu2014two}
M. Negreanu and J. Tello,
\newblock On a two species chemotaxis model with slow chemical diffusion.
\newblock {\em SIAM J. Math. Anal.}, 46(6): 3761--3781, 2014.

\bibitem{nirenberg1966extended}
L. Nirenberg,
\newblock An extended interpolation inequality.
\newblock {\em Ann. Sc. Norm. Super. Pisa, Cl. Sci.}, 20(4): 733--737, 1966.

\bibitem{painter2003modelling}
K. Painter and J. Sherratt,
\newblock Modelling the movement of interacting cell populations.
\newblock {\em J. Theoret. Biol.}, 225(3): 327--339, 2003.

\bibitem{pearce2007chemotaxis}
I. Pearce, M. Chaplain, P. Schofield, A. Anderson and S. Hubbard,
\newblock Chemotaxis-induced spatio-temporal heterogeneity in multi-species
  host-parasitoid systems.
\newblock {\em J. Math. Biol.}, 55(3): 365--388, 2007.

\bibitem{RachidiAsymptotic}
R. Salako and T. Issa,
\newblock Asymptotic dynamics in a two-species chemotaxis model with non-local
  terms.
\newblock {\em Discrete Contin. Dyn. Syst. Ser. B}, 22(10): 3839--3874, 2017.

\bibitem{sherratt2009boundedness}
J. Sherratt, S. Gourley, N. Armstrong and K. Painter,
\newblock Boundedness of solutions of a non-local reaction-diffusion model for
  adhesion in cell aggregation and cancer invasion.
\newblock {\em European J. Appl. Math.}, 20(1): 123--144, 2009.

\bibitem{stinner2014competitive}
C. Stinner, J. Tello and M. Winkler,
\newblock Competitive exclusion in a two-species chemotaxis model.
\newblock {\em J. Math. Biol.}, 68(7): 1607--1626, 2014.

\bibitem{szymanska2009mathematical}
Z. Szyma{\'n}ska, C. Rodrigo, M. Lachowicz and M. Chaplain,
\newblock Mathematical modelling of cancer invasion of tissue: the role and
  effect of nonlocal interactions.
\newblock {\em Math. Models Methods Appl. Sci.},
  19(02): 257--281, 2009.

\bibitem{tao2015boundedness}
Y.S. Tao and M. Winkler,
\newblock Boundedness vs. blow-up in a two-species chemotaxis system with two
  chemicals.
\newblock {\em Discrete Contin. Dyn. Syst. Ser. B}, 204(9): 3165-3183, 2015.

\bibitem{tao2016blow}
Y.S. Tao and M. Winkler,
\newblock Blow-up prevention by quadratic degradation in a two-dimensional
  Keller-Segel-Navier-Stokes system.
\newblock {\em Z. Angew. Math. Phy.},
  67(6): 138, 2016.

\bibitem{tello2007chemotaxis}
J. Tello and M. Winkler,
\newblock A chemotaxis system with logistic source.
\newblock {\em Comm. Partial Differential Equations},
  32(6): 849--877, 2007.

\bibitem{tello2012stabilization}
J. Tello and M. Winkler,
\newblock Stabilization in a two-species chemotaxis system with a logistic
  source.
\newblock {\em Nonlinearity}, 25(5): 1413, 2012.

\bibitem{wolansky2002multi}
G. Wolansky,
\newblock Multi-components chemotactic system in the absence of conflicts.
\newblock {\em European J. Appl. Math.}, 13(6): 641--661, 2002.

\bibitem{wang2019improvement}
L.C. Wang,
\newblock Improvement of conditions for boundedness in a two-species chemotaxis
  competition system of parabolic-parabolic-elliptic type.
\newblock {\em J. Math. Anal. Appl.}, 484(1): 123705, 2020.

\bibitem{winkler2010boundedness}
M. Winkler,
\newblock Boundedness in the higher-dimensional parabolic-parabolic chemotaxis system with logistic source.
\newblock {\em Comm. Partial Differential Equations}, 35(8): 1516--1537, 2010.

\bibitem{winkler2011blow}
M. Winkler,
\newblock Blow-up in a higher-dimensional chemotaxis system despite logistic
  growth restriction.
\newblock {\em J. Math. Anal. Appl.},
  2(384): 261--272, 2011.

\bibitem{winkler2013finite}
M. Winkler,
\newblock Finite-time blow-up in the higher-dimensional parabolic--parabolic
  Keller-Segel system.
\newblock {\em J. Math. Pures Appl.},
  100(5): 748--767, 2013.

\bibitem{Winkler2014How}
M. Winkler,
\newblock How far can chemotactic cross-diffusion enforce exceeding carrying
  capacities?
\newblock {\em J. Nonlinear Sci.}, 24(5): 809--855, 2014.

\bibitem{winkler2017emergence}
M. Winkler,
\newblock Emergence of large population densities despite logistic growth
  restrictions in fully parabolic chemotaxis systems.
\newblock {\em Discrete Contin. Dyn. Syst. Ser. B}, 22(7): 2777--2793, 2017.

\bibitem{winkler2018finite}
M. Winkler,
\newblock Finite-time blow-up in low-dimensional Keller-Segel systems with
  logistic-type superlinear degradation.
\newblock {\em Z. Angew. Math. Phy.}, 69(2): 40,
  2018.

\bibitem{Wang2015Time}
Q. Wang, J.Y. Yang and L. Zhang,
\newblock Time-periodic and stable patterns of a two-competing-species
  Keller-Segel chemotaxis model: effect of cellular growth.
\newblock {\em Discrete Contin. Dyn. Syst. Ser. B}, 22(9): 3547--3574, 2017.

\bibitem{xiang2018strong}
T. Xiang,
\newblock How strong a logistic damping can prevent blow-up for the minimal
  Keller-Segel chemotaxis system?
\newblock {\em J. Math. Anal. Appl.},
  459(2): 1172--1200, 2018.

\bibitem{muxu}
G.Y. Xu,
\newblock The carrying capacity to chemotaxis system with two species and competitive kinetics in $N$ dimensions.
\newblock {\em Z. Angew. Math. Phy.}, DOI: 10.1007/s00033-020-01363-z, in press.

\bibitem{yang2015boundedness}
C.B. Yang, X.R. Cao, Z.X. Jiang and S.N. Zheng,
\newblock Boundedness in a quasilinear fully parabolic Keller-Segel system of
  higher dimension with logistic source.
\newblock {\em J. Math. Anal. Appl.},
  430(1): 585--591, 2015.

\bibitem{yao2006chemotaxis}
J. Yao and C. Allen,
\newblock Chemotaxis is required for virulence and competitive fitness of the
  bacterial wilt pathogen ralstonia solanacearum.
\newblock {\em J. Bacteriol.}, 188(10): 3697--3708, 2006.

\end{thebibliography}

\end{document}